\documentclass[11pt, a4paper]{article}
\usepackage{graphicx}
\usepackage{epstopdf}
\usepackage{amsthm,amsmath,amssymb}
\usepackage{stackrel}
\usepackage{ulem}
\newtheorem{thm}{Theorem}[section]
\newtheorem{Lem}{Lemma}[section]

\newtheorem{defn}{Definition}[section]
\newtheorem{exam}{Example}[section]
\newtheorem{rem}{Remark}[section]

\begin{document}
\begin{center}
\LARGE {\bf {\boldmath Ry\={u}\={o} } Nim: A Variant of the classical game of Wythoff Nim}
\end{center}

\begin{center}
\large Ryohei Miyadera \footnote{Kwansei Gakuin High School}, Yuki Tokuni \footnote{Kwansei Gakuin University},Yushi Nakaya \footnote{Tohoku University}, Masanori Fukui \footnote{Hyogo University of Teacher Education}, Tomoaki Abuku \footnote{University of Tsukuba}, Koki Suetsugu \footnote{Kyoto University}
\end{center}

\begin{abstract}

The authors introduce the impartial game of the generalized Ry\={u}\={o} Nim, a variant of the classical game of Wythoff Nim. In the latter game, two players take turns in moving a single queen on a large chessboard, attempting to be the first to put her in the upper left corner, position $(0,0)$. 
Instead of the queen used in Wythoff Nim, we use the generalized Ry\={u}\={o} for a given natural number $p$. The generalized Ry\={u}\={o} for $p$ can be moved 
horizontally and vertically, as far as one wants. It also can be moved diagonally from $(x,y)$ to $(x-s,y-t)$, where $s,t$ are non-negative integers such that $1 \leq s \leq x, 1 \leq t \leq y \textit{ and } s+t \leq p-1$. When $p$ is $3$, the generalized Ry\={u}\={o} for $p$ is a Ry\={u}\={o} , i.e., a promoted hisha piece of Japanese chess. A Ry\={u}\={o} combines the power of the rook and the king in Western chess. 
The generalized Ry\={u}\={o} Nim for $p$ is mathematically the same as the Nim with two piles of counters in which a player may take any number from either heap, and a player may also simultaneously remove $s$ counters from either of the piles and $t$ counters from the other, where $s+t \leq p-1$ and $p$ is a given natural number.
The Grundy number of the generalized Ry\={u}\={o} Nim for $p$ is given by 
$\bmod(x+y,p) + p(\lfloor \frac{x}{p} \rfloor	\oplus \lfloor \frac{y}{p}\rfloor)$.
The authors also study the generalized Ry\={u}\={o} Nim for $p$ with a pass move. The generalized Ry\={u}\={o} Nim for $p$ without a pass move has simple
formulas for Grundy numbers. This is not the case after the introduction of a pass move, but it still has simple formulas for the previous player's positions. We also study the Ry\={u}\={o} Nim that restricted the diagonal and side movement. Moreover, we extended the Ry\={u}\={o} Nim dimension to the $n$-dimension.
\end{abstract}

\noindent
\textbf{Keyword}: Nim, Wythoff Nim, Grundy Number, a Pass Move, n-Dimention

\section{Generalized {\boldmath Ry\={u}\={o} } (dragon king) Nim}\label{generalnim}
Let $\mathbb{Z}_{\geq 0}$ be the set of non-negative integers and $N$ be the set of natural numbers. Let $p$ be a fixed natural number. For any $x \in \mathbb{Z}_{\geq 0}$, let $\bmod(x,p)$ denote the remainder obtained when $x $ is divided by $p$. 
Here, we introduce the impartial game of the generalized Ry\={u}\={o} Nim for $p$, a variant of the classical game of Wythoff Nim that was studied in \cite{wythoffpaper}. 	
Instead of the queen used in Wythoff Nim, we use the generalized Ry\={u}\={o} for $p$. When $p$ is $3$, the generalized Ry\={u}\={o} for $p$ is a Ry\={u}\={o} 	(dragon king ) of Japanese chess. A Ry\={u}\={o} combines the power of the rook and the king in Western chess.
Let us break with chess traditions here and denote fields on the chessboard by pairs of numbers. The field in the upper left corner will be denoted by $(0,0)$, and the other ones will be denoted according to a Cartesian scheme: field $(x, y)$ denotes $x$ fields to the right followed by $y$ fields down (see Figure \ref{narihishacoordinate}).

\begin{defn}\label{geryuohnim}	
We define the generalized Ry\={u}\={o} Nim for $p$.
The generalized Ry\={u}\={o} for $p$ is placed on a chessboard of unbounded size, and two players move it in turns.
The generalized Ry\={u}\={o} for $p$ is to be moved to the left or upwards, vertically, as far as one wants. It can also be moved diagonally from $(x,y)$ to
$(x-s,y-t)$, where $s,t$ are non-negative integers such that $1 \leq s \leq x, 1 \leq t \leq y \textit{ and } s+t \leq p-1$. The generalized Ry\={u}\={o} for $p$ has to be moved by at least one field in each move.
The object of the game is to move the generalized Ry\={u}\={o} for $p$ to the "winning field" $(0,0)$; whoever moves the generalized Ry\={u}\={o} for $p$ to this field wins the game.
\end{defn}

\begin{defn}\label{gryuohnim2}	
Suppose that there are two piles of counters. Players take turns removing counters from one or both piles. A player may take any number from either heap, and a player may also simultaneously remove $n$ counters from either of the piles and $m$ from the other, where $n+m \leq p-1$ and $p$ is a given natural number.
\end{defn}

\begin{rem}
Moving to the left or upwards, or to the upper left on the chessboard, is mathematically the same as 
taking counters from either heap or some counters from both, and hence, 
the game defined in Definition \ref{geryuohnim}	 is mathematically the same as that defined in Definition \ref{gryuohnim2}.
\end{rem}

\begin{figure}[!htb]
\begin{center}
\includegraphics[height=0.3\columnwidth,bb=0 0 273 265]{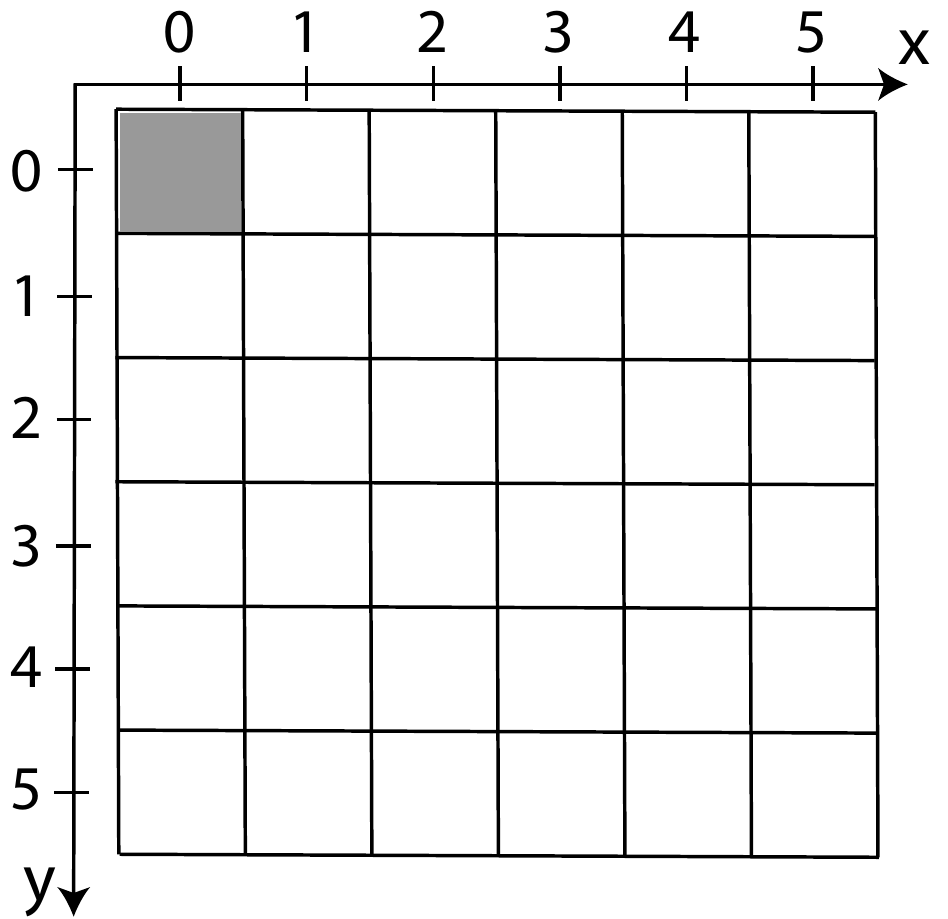}
\label{narihishacoordinate}
\caption{Definition of coordinates}
\end{center}
\end{figure}

In this article, we only treat impartial games (see \cite{lesson} or \cite{combysiegel} for a background on impartial games).
For an impartial games without draws, there will only be two outcome classes.

\begin{defn}\label{NPpositions}
$(a)$ $\mathcal{N}$-positions, from which the next player can force a win, as long as he plays correctly at every stage.\\
$(b)$ $\mathcal{P}$-positions, from which the previous player $($the player who will play after the next player$)$ can force a win, as long as he plays correctly at every stage.
\end{defn}

We use the theory of Grundy numbers to study impartial games without draws. To define the Grundy numbers, we need some definitions.

\begin{defn}\label{moveceil}
For any position $\mathbf{p}$ of a game $G$, there is a set of positions that can be reached by making precisely one move in $G$, which we will denote by {\rm move}$(\mathbf{p})$. 
\end{defn}

The ${\rm move}$ of the generalized Ry\={u}\={o} for $p$ is defined in Definition \ref{defofmoveforgeneralg}.	

\begin{defn}\label{defofmoveforgeneralg} 
We define ${\rm move}((x,y))$ of the generalized Ry\={u}\={o} for $p$.
For $x,y \in \mathbb{Z}_{\geq 0}$ such that $x+y \geq 1$, let
\begin{align}
M_{g1}= \{(u,y):u<x\} \label{geryuuoumoveex1g} \\
M_{g2}= \{(x,v):v<y\}, \label{geryuuoumoveex2g}
\end{align}
where $u,v \in \mathbb{Z}_{\geq 0}$.
For $x,y \in \mathbb{Z}_{\geq 0}$ such that $1 \leq x,y$, let			
\begin{align}
M_{g3}= \{(x-s,y-t):1 \leq s \leq x, 1 \leq t \leq y \ and \ s+t \leq p-1\},\label{geryuuoumoveex3g}
\end{align}
where $s,t \in \mathbb{Z}_{\geq 0}$.\\
We define ${\rm move}((x,y))=M_{g1} \cup M_{g2} \cup M_{g3}.$
\end{defn}
\begin{rem}
The sets (\ref{geryuuoumoveex1g}), (\ref{geryuuoumoveex2g}), and (\ref{geryuuoumoveex3g}) denote horizontal, vertical, and 
upper left moves.
If $x=0$ or $y=0$, then $M_{g1}$ or $M_{g2}$ is empty, respectively. When $x=0$ or $y=0$, then $M_{g3}$ is empty.\\
$M_{g1}, M_{g2}$, and $M_{g3}$ depend on $(x,y)$; hence, mathematically, it is more precise to write $M_{g1}(x,y), M_{g2}(x,y)$, and $M_{g3}(x,y)$.
If we write $M_{g1}(x,y)$ instead of $M_{g1}$, the relations and equations appear more complicated. Therefore,
we omit $(x,y)$ for convenience. 
\end{rem}

\begin{exam}
Figure \ref{geryuuoumovep6}, Figure \ref{geryuuoumovep4} and Figure \ref{geryuuoumovep8} show the moves of the generalized Ry\={u}\={o} for 
$p = 3$, $p = 4$ and $p=8$, respectively. Figure \ref{geryuuoumovepg} presents the moves of the generalized Ry\={u}\={o} for a natural number 
$p$.
In these figures, the horizontal move (the set in (\ref{geryuuoumoveex1g}) ) and the vertical move (the set in (\ref{geryuuoumoveex2g}))
are denoted by dotted lines, and the upper left move (the set in (\ref{geryuuoumoveex3g})) is denoted by a set of small circles.
\begin{figure}[!htb]
\begin{minipage}[!htb]{0.45\columnwidth}
\begin{center}
\includegraphics[height=0.75\columnwidth,bb=0 0 525 525]{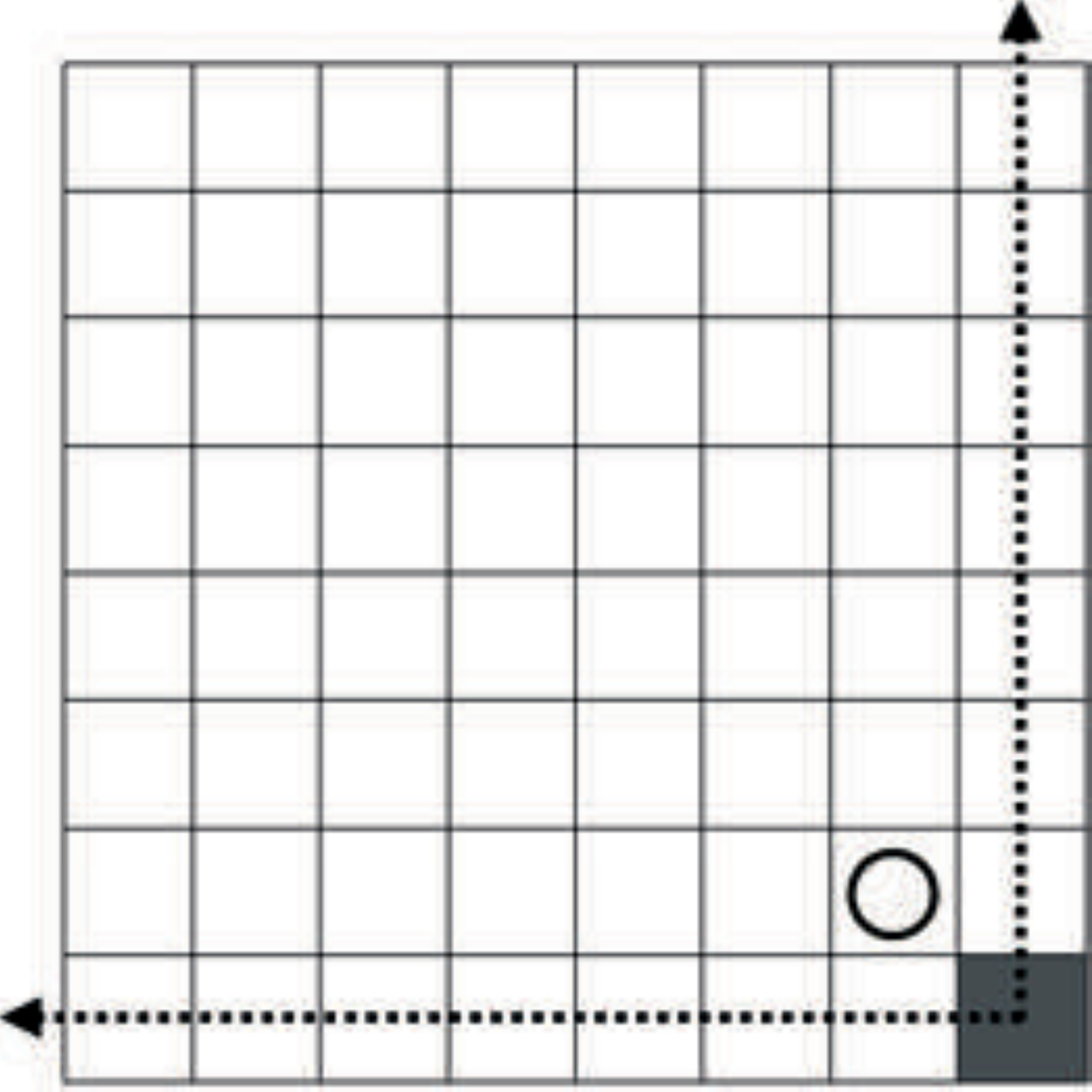}
\caption{Moves of the Ry\={u}\={o} }
\label{geryuuoumovep6}
\end{center}
\end{minipage}
\begin{minipage}[!htb]{0.45\columnwidth}
\begin{center}
\includegraphics[height=0.75\columnwidth, bb=0 0 525 525]{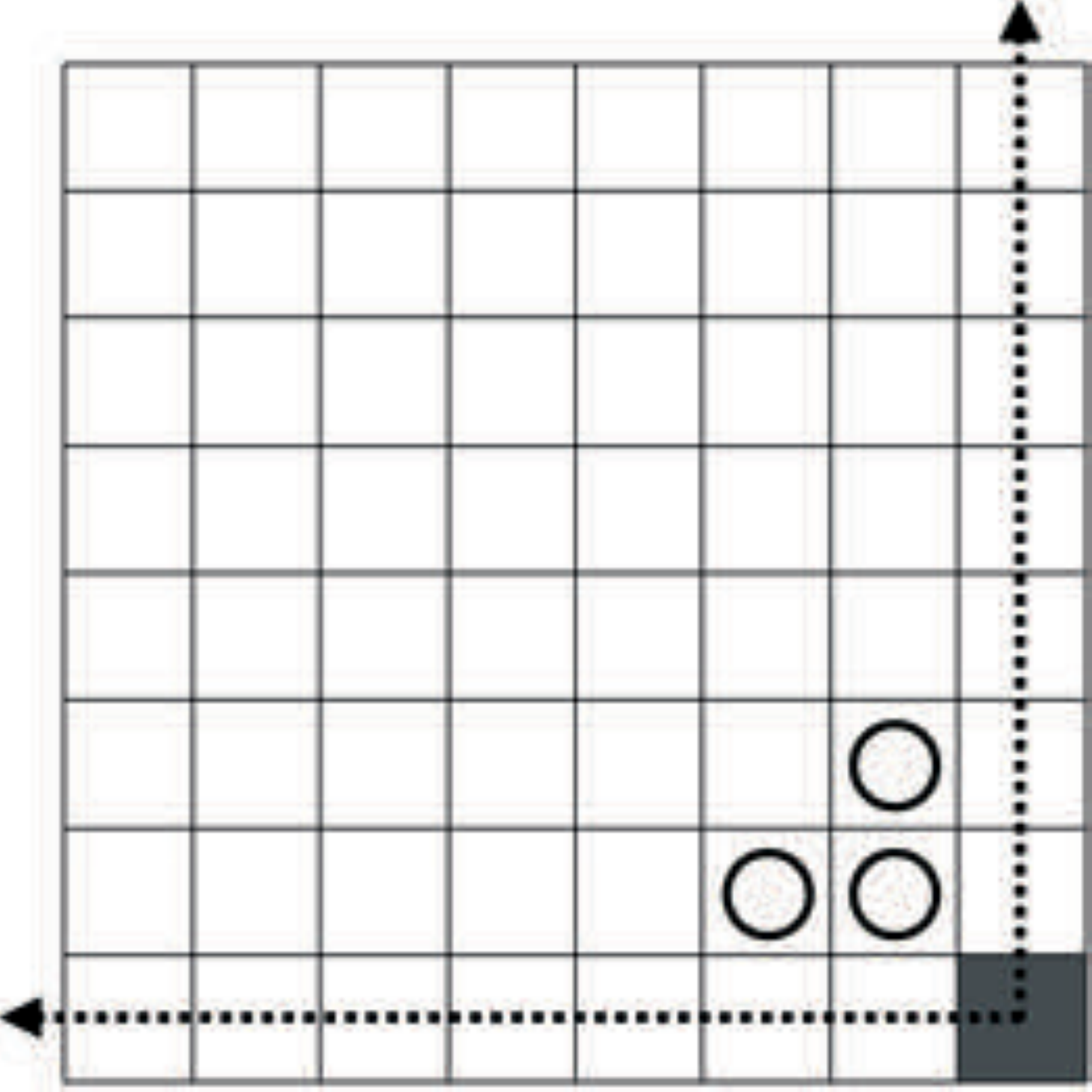}
\caption{Moves of a general Ry\={u}\={o} for $p=4$} 
\label{geryuuoumovep4}
\end{center}
\end{minipage}
\end{figure}
\begin{figure}[!htb]
\begin{minipage}[!htb]{0.45\columnwidth}
\begin{center}
\includegraphics[height=0.75\columnwidth, bb=0 0 525 525]{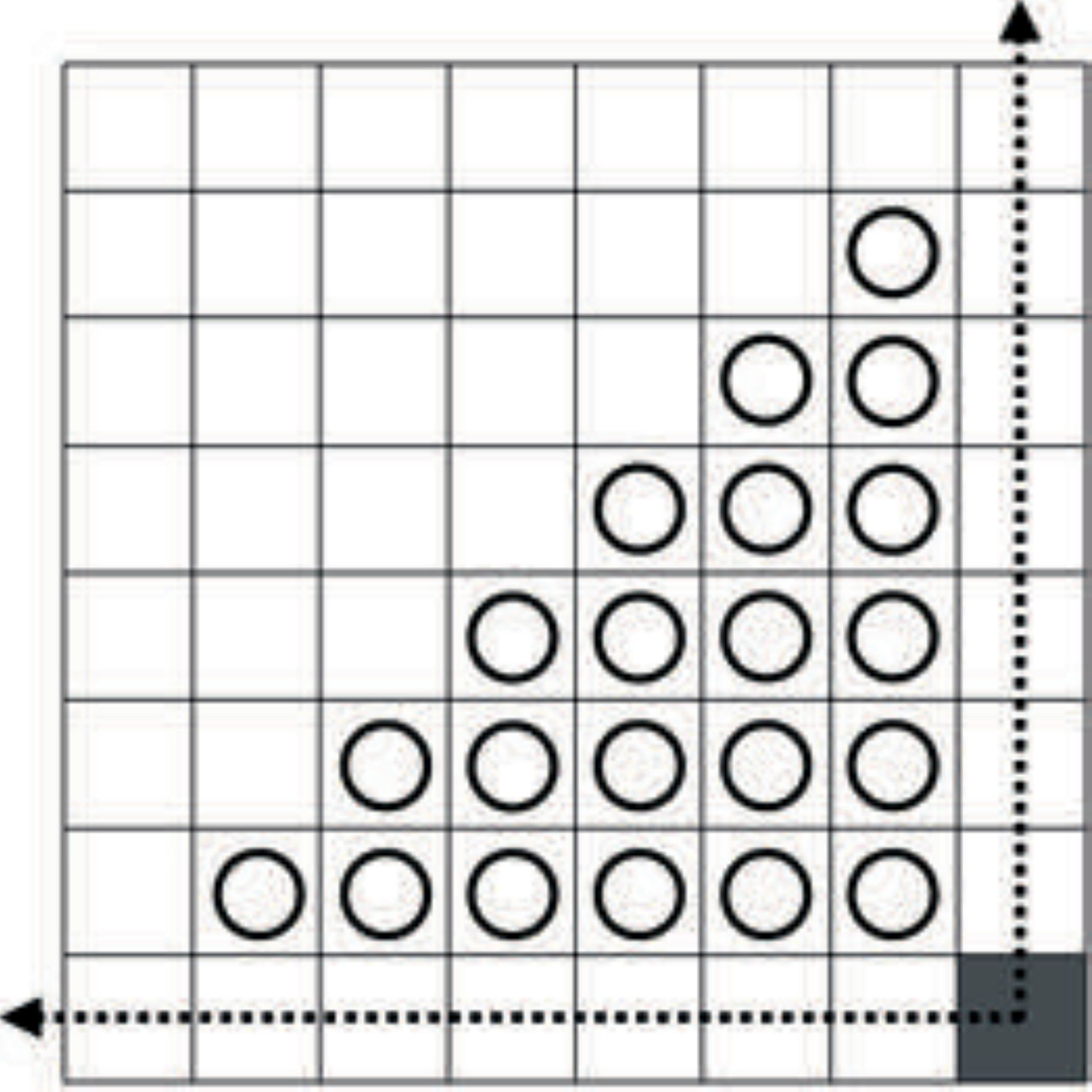}
\caption{Moves of a general Ry\={u}\={o} for $p=8$}
\label{geryuuoumovep8}
\end{center}
\end{minipage}
\begin{minipage}[!htb]{0.45\columnwidth}
\begin{center}
\includegraphics[height=0.7\columnwidth,bb=0 0 275 242]{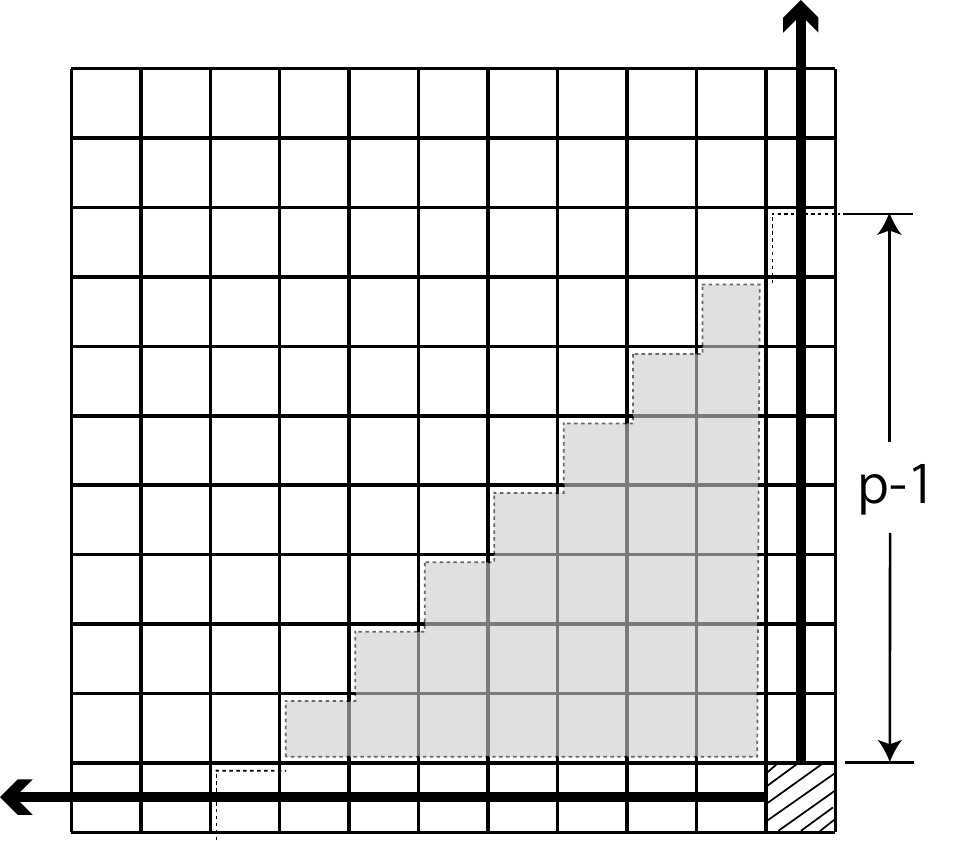}
\caption{Moves of a general Ry\={u}\={o} for $p$}
\label{geryuuoumovepg}
\end{center}
\end{minipage}
\end{figure}
\end{exam}

\begin{defn}\label{defofmexgrundy}
$(i)$ The \textit{minimum excluded value} $({\rm mex})$ of a set, $S$, of non-negative integers is the least non-negative integer that is not in S. \\
$(ii)$ Each position $\mathbf{p}$ of an impartial game G has an associated Grundy number, and we denote it by $\mathcal{G}(\mathbf{p})$.\\
The Grundy number is calculated recursively: 
$\mathcal{G}(\mathbf{p}) = {\rm mex}\{\mathcal{G}(\mathbf{h}): \mathbf{h} \in move(\mathbf{p})\}.$
\end{defn}

\begin{exam}
Examples of calculation of {\rm mex}.
\begin{align*}
&{\rm mex} \{0,1,2,3\}=4, \ {\rm mex} \{1,1,2,3\}=0, \\
&{\rm mex} \{0,2,3,5\}=1, \ and \ {\rm mex} \{0,0,0,1\}=2. 
\end{align*}
\end{exam}

\begin{thm}\label{theoremofsumg}
Let $\mathcal{G}$ be the $Grundy$ number. Then, 
$\mathbf{h}$ is a $\mathcal{P}$-$position$ if and only if $\mathcal{G}(\mathbf{h})=0$. 
\end{thm}
For the proof of this theorem, see $\cite{lesson}$.

We define nim-sum that is important for combinatorial game theory. 
\begin{defn}\label{definitionfonimsum11}
Let $x,y$ be non-negative integers, and write them in base $2$, so $x = \sum\limits_{i = 0}^n {{x_i}} {2^i}$ and $y = \sum\limits_{i = 0}^n {{y_i}} {2^i}$ with ${x_i},{y_i} \in \{ 0,1\}$.\\
We define the nim-sum $x \oplus y$ by
\begin{align}
x \oplus y = \sum\limits_{i = 0}^n {{w_i}} {2^i},
\end{align}
where $w_{i}=x_{i}+y_{i} \ (\bmod\ 2)$.
\end{defn}
	
In this section we prove that the Grundy number of the generalized Ry\={u}\={o} Nim for $p$ 
is $\mathcal{G}((x,y)) = \bmod(x+y,p) + p(\lfloor \frac{x}{p} \rfloor \oplus \lfloor \frac{y}{p}\rfloor)$ in Theorem \ref{theofsufficientcondition}. However, we need some lemmas to prove Theorem \ref{theofsufficientcondition}.	
The calculations used in these lemmas may not be easy to understand for some people. Hence, we study Example \ref{exampleforproof}.
You can skip Example \ref{exampleforproof}, and move to Lemma \ref{lemmafornimsum}.
	
\begin{exam}\label{exampleforproof}
Let $p=3$. We prove (\ref{grundyformula1}) for $(x,y)=(17,19)$ using mathematical induction.
\begin{align}\label{grundyformula1}
 \mathcal{G}((x,y)) = \bmod(x+y,p) + p(\lfloor \frac{x}{p} \rfloor \oplus \lfloor \frac{y}{p}\rfloor).
\end{align}
We assume that (\ref{grundyformula1}) is valid for $(x,y)$, such that $x \leq 17$, $y \leq 19$ and, $x+y < 36$ and we prove 
\begin{align}\label{grundyformula1719}
\mathcal{G}((x,y)) = \bmod(36,3) + 3(5 \oplus6) .
\end{align}
By the definition of Grundy number
\begin{align}\label{g1719}
\mathcal{G}((17,19)) = {\rm mex} \{\mathcal{G}((u,v)):(u,v) \in {\rm move}((17,19))\}.
\end{align}
By Definition \ref{defofmoveforgeneralg}, $move ((17,19))$ is the union the sets $M_{g1}$, $M_{g2}$ and $M_{g3}$.
	
\begin{align}
	M_{g1}&= \{(16,19), (15,19), (14,19), \cdots ,(0,19) \}. \label{move1g} \\
	M_{g2}&= \{(17,18), (17,17),(17,16),\cdots ,(17,0) \}. \label{move2g} \\
	M_{g3}&= \{(16,18)\}. \label{move3g}
\end{align}
By the hypothesis of mathematical induction, the set $\{\mathcal{G}((u,v)) \in M_{g1} \}$ is the union of the sets in (\ref{mg11}) and (\ref{mg12}).
\begin{align}
 & \{ \bmod(35,3) + 3(\lfloor \frac{16}{3} \rfloor \oplus \lfloor \frac{19}{3}\rfloor), \bmod(34,3)+3(\lfloor \frac{15}{3} \rfloor \oplus \lfloor \frac{19}{3}\rfloor) \} \nonumber \\
&= \{2+ 3(5 \oplus 6), 1 + 3(5 \oplus 6)\}. \label{mg11}\\
& \{ \bmod(33,3) +3(\lfloor \frac{14}{3} \rfloor \oplus \lfloor \frac{19}{3}\rfloor),
	\bmod(32,3) + 3(\lfloor \frac{13}{3} \rfloor \oplus \lfloor \frac{19}{3}\rfloor), \nonumber \\
 & \cdots ,\bmod(19,3)+3(\lfloor \frac{0}{3} \rfloor \oplus \lfloor \frac{19}{3}\rfloor) \} \nonumber \\
 &=\{3(4 \oplus 6), 2 + 3(4 \oplus 6),1 + 3(4 \oplus 6), 3(3 \oplus 6), 2 + 3(3 \oplus 6),1 + 3(3 \oplus 6), \nonumber \\
&\cdots ,3 (0 \oplus 6), 2 + 3(0 \oplus 6), 1+ 3(0 \oplus 6) \}. \label{mg12}
	\end{align}
	
By the hypothesis of mathematical induction, the set $\{\mathcal{G}((u,v)) \in M_{g2} \}$ is the union of the sets in (\ref{mg21}) and (\ref{mg22}).
\begin{align}
& \{ \bmod(35,3) + 3(\lfloor \frac{17}{3} \rfloor \oplus \lfloor \frac{18}{3}\rfloor)\}= \{ 2 + 3(5 \oplus 6) \}. \label{mg21} \\
& \{ \bmod(34,3) + 3(\lfloor \frac{17}{3} \rfloor \oplus \lfloor \frac{17}{3}\rfloor), \cdots ,\bmod(17,3) + 3(\lfloor \frac{17}{3} \rfloor \oplus \lfloor \frac{0}{3}\rfloor) \} \nonumber \\
& =\{1 + 3(5 \oplus 5), 3(5 \oplus 5),2 + 3(5 \oplus 5), 1 + 3(5 \oplus 4), 3(5 \oplus 4),\nonumber \\
& 2 + 3(5 \oplus 4), \cdots , 1 + 3(5 \oplus 0), 3(5 \oplus 0), 2 + 3(5 \oplus 0)\}. \label{mg22}
\end{align}
	
By the hypothesis of mathematical induction, the set $\{\mathcal{G}((u,v)) \in M_{g3} \}$ is the set in (\ref{mg31}).
\begin{align}
\{ \bmod(34,3) + 3(\lfloor \frac{16}{3} \rfloor \oplus \lfloor \frac{18}{3}\rfloor)\} 
= \{1 + 3(5 \oplus 6) \}. \label{mg31}
\end{align}
	
We denote the union of the sets in (\ref{mg12}) and (\ref{mg22}) by $A_{3,6}$. 
Note that this union of sets is $A_{k,h}$ for $k=5$ and $h=6$ used in Lemma \ref{lemmafornimsum}.\\
Then,
\begin{align}	
&A_{3,6}= \bigcup_{u=0}^{2}\{3((5-t)\oplus 6)+u:t=1,2,\cdots ,5 \} \cup \bigcup_{u=0}^{2}\{3(5 \oplus (6-t))+u\nonumber\\
&:t=1,2,\cdots ,6 \}.\label{mg313}
\end{align}
	
We denote by $C_{5,6,2,1}$ the union of the sets in (\ref{mg11}), (\ref{mg21}) \ and\ (\ref{mg31}).
Note that this union of sets is $C_{k,h,v,w}$ for $k=5,h=6,v=2,w=1$ used in Lemma \ref{lemmaforgrundyof}.
Then
\begin{align}
C_{5,6,2,1} = \{ 3(5 \oplus 6)+w: w = 1,2 \}.\label{mg313b}
\end{align}
	
We need to prove (\ref{grundyformula1719}).
\begin{align}
& \bmod(36,3) + 3(5 \oplus 6) \nonumber \\
&= {\rm mex}(\bigcup_{u=0}^{2}\{3((5-t)\oplus 6)+u:t=1,2,\cdots ,5 \} \nonumber \\
& \cup \bigcup_{u=0}^{2}\{3(5 \oplus (6-t))+u:t=1,2,\cdots ,6 \} \nonumber \\
& \cup \{ 3(5 \oplus 6)+w: w = 1,2 \}).\label{mg313c}
\end{align}
We need the following lemmas to prove (\ref{mg313c}).
\end{exam}

\begin{Lem}\label{lemmafornimsum}
Let $k,h \in \mathbb{Z}_{\geq0}$. Then,
\begin{align}
k \oplus h = {\rm mex}(\{(k-t)\oplus h:t=1,2,\cdots ,k \} \cup \{k \oplus (h-t):t=1,2,\cdots ,h \}).
\end{align}
\end{Lem}
\begin{proof}
We omit the proof, because this is a well-known fact about Nim sum $\oplus$ (see Proposition 1.4. (p.181) of \cite{combysiegel}).
\end{proof}
\begin{Lem}\label{nimby3}
Let $A_{k,h} = \bigcup_{u=0}^{p-1}\{p((k-t)\oplus h)+u:t=1,2,\cdots ,k \}
\cup \bigcup_{u=0}^{p-1}\{p(k \oplus (h-t))+u:t=1,2,\cdots ,h \}$ for $k,h \in \mathbb{Z}_{\geq0}$. Then, the following statements are true.\\
$(a)$ For any $v=1,\cdots ,p-1$,
\begin{align}\label{pkhv}
p(k \oplus h)+v= {\rm mex}(A_{k,h} \cup \{ p(k \oplus h)+w:w=0,\cdots ,v-1\} ).
\end{align}
$(b)$ The argument in $(a)$ is valid for $v=0$, and we get $p(k \oplus h)= {\rm mex}(A_{k,h}).$
\end{Lem}
\begin{proof}
$(a)$ By Lemma \ref{lemmafornimsum} and Definition \ref{defofmexgrundy} (the definition of ${\rm mex}$)
\begin{align}\label{nimnotequal}
k \oplus h \notin \{(k-t)\oplus h:t=1,2,\cdots ,k \} \cup \{k \oplus (h-t):t=1,2,\cdots ,h \},
\end{align}
and hence, for any $v = 0,1,\cdots ,p-1$,
\begin{align}\label{nimnotequa2}
p(k \oplus h)+v \notin \{p((k-t)\oplus h)+v:t=1,2,\cdots ,k \}\nonumber \\
\cup \{p(k \oplus (h-t))+v:t=1,2,\cdots ,h \}.
\end{align}

Since $p(k \oplus h)+v \neq p((k-t)\oplus h)+u, p(k \oplus (h-t))+u$ for $u \neq v$, 
\begin{align}
\label{notineq}
p(k \oplus h) + v \notin A_{k,h}. 
\end{align}

Therefore, $p(k \oplus h)+v \notin A_{k,h} \cup \{ p(k \oplus h)+w:w=0,\cdots ,v-1\} .$
Let $s$ be an arbitrary non-negative integer such that
\begin{align}
\label{onesmallerthan}
p(k \oplus h)+v > s \geq 0.
\end{align}

We prove that $s \in A_{k,h} \cup \{ p(k \oplus h)+w:w=0,\cdots ,v-1\}$.\\
If $s \in \{ p(k \oplus h)+w:w=0,\cdots ,v-1\}$, then 
\begin{align}
\label{belongsto1}
s \in A_{k,h} \cup \{ p(k \oplus h)+w:w=0,\cdots ,v-1\}. \nonumber
\end{align}
Suppose that $s \notin \{ p(k \oplus h)+w:w=0,\cdots ,v-1\}$.
We choose non-negative integers $t,u$ such that $s=pt+u$ and $0 \leq u \leq p-1$.
Clearly,
$k \oplus h > t \geq 0$, and 
by Lemma \ref{lemmafornimsum} and Definition \ref{defofmexgrundy} (the definition of {\rm mex}),
\begin{align}
t &\in \{(k-t)\oplus h:t=1,2,\cdots ,k \} \cup \{k \oplus (h-t):t=1,2,\cdots ,h \},
\end{align}
and 
\begin{align}\label{belongsto}
s&=pt+u \in \{p((k-t)\oplus h)+u:t=1,2,\cdots ,k \}\nonumber \\
&\cup \{p(k \oplus (h-t))+u:t=1,2,\cdots ,h \} \subset A_{k,h}. 
\end{align}
Therefore, by (\ref{notineq}), (\ref{onesmallerthan}), (\ref{belongsto1}), (\ref{belongsto}), 
and Definition \ref{defofmexgrundy} (the definition of ${\rm mex}$),
$ p(k \oplus h)+v= {\rm mex}(A_{k,h} \cup \{ p(k \oplus h)+w:w=0,\cdots ,v-1\} ).$\\
$(b)$ The argument in $(a)$ is valid for $v=0$, and hence, we get
$ p(k \oplus h)= {\rm mex}(A_{k,h}).$
\end{proof}

\begin{Lem}\label{lemmaformodp}
Let $x,k \in \mathbb{Z}_{\geq0}$.
If $0 \leq k < \bmod(x,p)$, then
\begin{align}\label{kbelongtoset}
k \in \{\bmod(x-r,p):1 \leq r \leq x \textit{ and } r \leq p-1 \}.
\end{align}	
\end{Lem}

\begin{proof}
Let $k \in \mathbb{Z}_{\geq0}$ such that $0 \leq k < \bmod(x,p)$. We consider two cases.\\
\underline{Case $(a)$} First, we suppose that $x \leq p-1$. 
Since $0 \leq k < \bmod(x,p) = x$, 
$k \in \{0,1,2, \cdots,x-1\}$ = $\{x-r:1 \leq r \leq x \}$ $=\{\bmod(x-r,p):1 \leq r \leq x \leq p-1 \}$
$ = \{\bmod(x-r,p):1 \leq r \leq x \textit{ and } r \leq p-1 \}$.\\
\underline{Case $(b)$} Second, we suppose that $x > p-1$. Then, there exist $q,w \in \mathbb{Z}_{\geq0}$ such that $q \leq p-1$ and $x = pw+q$.
Then, $0 \leq k < \bmod(x,p) = q$.
$k \in \{0,1,2,\cdots ,q-1 \}$
$= \{q-u:1 \leq u \leq q \}$
$=\{\bmod(x-u,p):1 \leq u \leq q \}$
$ \subset \{\bmod(x-u,p):1 \leq u \leq p-1 \}$
$= \{\bmod(x-r,p):1 \leq r \leq x \textit{ and } r \leq p-1 \}$,
where the last equation is implied by the inequality $x > p-1$.
\end{proof}

\begin{Lem}\label{lemmaformex}
Let $V$ be a subset of $\mathbb{Z}_{\geq0}$, and let 
$v \in \mathbb{Z}_{\geq0}$ such that 
\begin{align}\label{conditionformex}
v = {\rm mex}(V).
\end{align}	
If $W$ is a subset of $\mathbb{Z}_{\geq0}$ such that $V \subset W$ and $v \notin W $, then
$v = {\rm mex}(W)$.
\end{Lem}	
\begin{proof}
This lemma is directly obtained from the definition of {\rm mex} (Definition \ref{defofmexgrundy}).
\end{proof}

\begin{Lem}\label{lemmaforgrundyof}
Let $k,h,v,w \in \mathbb{Z}_{\geq0}$ such that $0 \leq v,w \leq p-1$,
and let 
\begin{align}
& C_{k,h,v,w} = \{p(k\oplus h)+\bmod(v+w-t,p):1 \leq t \leq v\},\nonumber\\
& \cup \{p(k\oplus h)+\bmod(v+w-t,p):1 \leq t \leq w\},\nonumber\\
& \cup \{p(\lfloor \frac{pk+v-s}{p} \rfloor	\oplus \lfloor \frac{ph+w-t}{p}\rfloor)+\bmod(v+w-s-t,p):\nonumber\\
& 1 \leq s,t \ and \ s+t \leq p-1 \}.
\end{align}	

Then, the following two statements are true.\\
$(a)$ $p(k\oplus h)+\bmod(v+w,p) \notin C_{k,h,v,w}$. \\
$(b)$ $p(k\oplus h)+u \in C_{k,h,v,w} $ for any $u \in \mathbb{Z}_{\geq0}$ such that 
\begin{align}\label{conditionforb}
0 \leq u < \bmod(v+w,p).
\end{align}		
\end{Lem}	
\begin{proof}
$(a)$ Let $t \in \mathbb{Z}_{\geq0}$ such $1 \leq t \leq v$. Then, $t \leq v \leq p-1$, and hence,
$\bmod(v+w,p) \neq \bmod(v+w-t,p)$. Similarly, for any $t \in \mathbb{Z}_{\geq0}$ such $1 \leq t \leq w$,
we have $\bmod(v+w,p) \neq \bmod(v+w-t,p)$.
Since $1 \leq s+t \leq p-1$,\\
$\bmod(v+w,p) \neq \bmod(v+w-s-t,p)$.
Therefore, $p(k\oplus h)+\bmod(v+w,p) \notin C_{k,h,v,w}$.\\
$(b)$ 
Let $	0 \leq u < \bmod(v+w,p).$
By Lemma \ref{lemmaformodp}, 
$u = \bmod(v+w-r,p)$ for some $r \in \mathbb{Z}_{\geq0}$ such that $1 \leq r \leq v+w \textit{ and } r \leq p-1$.
Then, there exist $s,t \in \mathbb{Z}_{\geq0}$ such that $ s \leq v, t \leq w \textit{ and } r=s+t$, and
$u = \bmod(v-s+w-t,p)$ and $s+t \leq p-1.$ Here, we consider three cases.\\
\underline{Case $(b.1)$} Suppose that $s=0$. Then, 
$p(k\oplus h)+u$\\
$ = p(k\oplus h)+\bmod(v+w-t,p) \in C_{k,h,v,w} $.\\
\underline{Case $(b.2)$} Suppose that $t=0$. Then, 
$p(k\oplus h)+u$\\
$ = p(k\oplus h)+\bmod(v+w-s,p) \in C_{k,h,v,w} $.\\
\underline{Case $(b.3)$} Suppose that $1 \leq s,t$. Then,
$p(k\oplus h)+u$
$ = p(k\oplus h)+\bmod(v-s+w-t,p)$ \\ $ = p(\lfloor \frac{pk+v-s}{p} \rfloor	\oplus \lfloor \frac{ph+w-t}{p}\rfloor)+\bmod(v-s+w-t,p) \in C_{k,h,v,w}. $
\end{proof}
\begin{thm}\label{theofsufficientcondition}
The Grundy number of the generalized Ry\={u}\={o} 	 Nim for $p$ is 
\begin{align}\label{formulagrundynarihisha}
\mathcal{G}((x,y)) = \bmod(x+y,p) + p(\lfloor \frac{x}{p} \rfloor	\oplus \lfloor \frac{y}{p}\rfloor).
\end{align}
Here, $\bmod(x+y,p)$ is the remainder obtained when $x+y$ is divided by p.
\end{thm}
\begin{proof}
We prove by mathematical induction. We assume that 
Equation (\ref{formulagrundynarihisha}) is valid for 
$(u,v)$ when $u<x$ or $v<y$.
Let $(x,y)=(pk+v,ph+w)$ with $0 \leq v \leq p-1$ and $0 \leq w \leq p-1$.

\begin{align}
& move((x,y)) \nonumber \\
& = \{(pk+v-t,ph+w):1 \leq t \leq v\}\label{vpart} \\
& \cup \{(pk+v,ph+w-t):1 \leq t \leq w \}\label{wpart} \\
& \cup \{(pk-t,ph+w):t=1,2,\cdots ,pk\}\nonumber\\
& \cup \{(pk+v,ph-t):t=1,2,\cdots ,ph\} \nonumber\\
& \cup \{(pk+v-s,ph+w-t):1 \leq s,t \textit{ and } s+t \leq p-1\}. \nonumber 
\end{align}

If $v = 0$ or $w=0$, the set (\ref{vpart}) or the set (\ref{wpart}) is empty, respectively. 
By the definition of the Grundy number 
\begin{align}
& \mathcal{G}((x,y)) = {\rm mex}(\{\mathcal{G}((pk+v-t,ph+w)):1 \leq t \leq v \}\label{ggrundyfirstp1} \\
& \cup \{\mathcal{G}((pk+v,ph+w-t)):1 \leq t \leq w\} \label{ggrundyfirstp2} \\
& \cup \{\mathcal{G}((pk-t,ph+w)):t=1,2,\cdots ,pk\} \label{ggrundyfirstp3} \\
& \cup \{\mathcal{G}((pk+v,ph-t)):t=1,2,\cdots ,ph\} \label{ggrundyfirstp4} \\
& \cup \{\mathcal{G}((pk+v-s,ph+w-t)):\label{ggrundyfirstp5} \\
& 1 \leq s,t \textit{ and } s+t \leq p-1\}). \nonumber
\end{align}
Next, we study the set of Grundy numbers in (\ref{ggrundyfirstp3}) and (\ref{ggrundyfirstp4}). 
Since 
\begin{align*}
&\{\bmod(v,p),\bmod(v+1,p),\cdots ,\bmod(v+p-1,p)\} = \{\bmod(w,p),\bmod(w+1,p), \cdots,\\
&\bmod(w+p-1,p)\} = \{0,1,\cdots ,p-1\},
\end{align*}
by the assumption of mathematical induction, we have
\begin{align}
&\{\mathcal{G}((pk-t,ph+w)):t=1,2,\cdots ,pk\} \cup \{\mathcal{G}((pk+v,ph-t)):t=1,2,\cdots ,ph\}\nonumber\\
& = \{\bmod(pk-t+ph+w,p) + p(\lfloor \frac{pk-t}{p} \rfloor	\oplus \lfloor \frac{ph+w}{p}\rfloor): t = 1,2,\cdots ,pk\} \nonumber\\
& \cup \{\bmod(pk+v+ph-t,p) + p(\lfloor \frac{pk+v}{p} \rfloor	\oplus \lfloor \frac{ph-t}{p}\rfloor): t=1,2,\cdots ,ph \} \nonumber\\
& = \{p((k-1)\oplus h)+\bmod(w+p-1,p), p((k-1) \oplus h) + \bmod(w+p-2,p),\nonumber\\
&\cdots ,p((k-1) \oplus h)\nonumber\\
&+\bmod(w,p),\cdots ,p(0 \oplus h) +\bmod(w+p-1,p),p(0 \oplus h) +\bmod(w+p-2,p),\nonumber\\
&\cdots , p(0 \oplus h)\nonumber\\
&+\bmod(w,p)\} \cup \{ p(k \oplus (h-1))+\bmod(v+p-1,p),p(k \oplus (h-1)) \nonumber\\
&+\bmod(v+p-2,p),\cdots ,p(k \oplus (h-1))+\bmod(v,p),\cdots ,p(k \oplus 0)\nonumber\\
&+\bmod(v+p-1,p), p(k \oplus 0)+\bmod(v+p-2,p),\cdots ,p(k \oplus 0)+\bmod(v,p)\}\nonumber \\
& = \bigcup_{u=0}^{p-1}\{p((k-t)\oplus h)+u:t=1,2,\cdots ,k \} \nonumber\\
&\cup \bigcup_{u=0}^{p-1}\{p(k \oplus (h-t))+u:t=1,2,\cdots ,h \} = A_{k,h}. \label{setofG1}
\end{align}
Here, $A_{k,h}$ is the set defined in Lemma \ref{nimby3}.
Next, we study the set of Grundy numbers in (\ref{ggrundyfirstp1}), (\ref{ggrundyfirstp2}), and (\ref{ggrundyfirstp5}).
By the assumption of mathematical induction, 
\begin{align}
& \{\mathcal{G}((pk+v-t,ph+w)):1 \leq t \leq v \} \cup \{\mathcal{G}((pk+v,ph+w-t)):1 \leq t \leq w \}\nonumber \\
& \cup \{\mathcal{G}((pk+v-1,ph+w-1))\}\nonumber \\
& = \{p(k\oplus h)+\bmod(v-t+w,p):1 \leq t \leq v \} \cup \{p(k\oplus h)+\bmod(v+w-t,p):\nonumber\\
&1 \leq t \leq w \}\nonumber\\
& \cup \{p(\lfloor \frac{pk+v-s}{p} \rfloor	\oplus \lfloor \frac{ph+w-t}{p}\rfloor)+\bmod(v-s+w-t,p): 1 \leq s,t \nonumber\\
&\textit{ and } s+t \leq p-1 \} = C_{k,h,v,w}.\label{lastofsets}
\end{align}
Note that $C_{k,h,v,w}$ is used in Lemma \ref{lemmaforgrundyof}.
By (\ref{ggrundyfirstp1}), (\ref{ggrundyfirstp2}), (\ref{ggrundyfirstp3}), (\ref{ggrundyfirstp4}), (\ref{ggrundyfirstp5}), (\ref{setofG1}), and (\ref{lastofsets}). 
\begin{align}\label{grundyequaltoAC}
\mathcal{G}((x,y)) = {\rm mex}(A_{k,h} \cup C_{k,h,v,w}).
\end{align}
By Lemma \ref{lemmaforgrundyof},
\begin{align}\label{Cbigger3k3h}
C_{k,h,v,w} \supset \{p(k\oplus h)+u: 0 \leq u < \bmod(v+w,p)\}.
\end{align}
By Lemma \ref{lemmaforgrundyof},
$p(k\oplus h)+\bmod(v+w,p) \notin C_{k,h,v,w}$, 
and it is clear that 
$p(k\oplus h)+\bmod(v+w,p) \notin A_{k,h}$.
Therefore, 
\begin{align}\label{notbelongCA}
p(k\oplus h)+\bmod(v+w,p) \notin A_{k,h} \cup C_{k,h,v,w}.
\end{align}
By Lemma \ref{nimby3},
\begin{align}\label{lastequation}
&p(k\oplus h)+ \bmod(v+w,p) = {\rm mex}(A_{k,h} \cup \{p(k\oplus h)+u: 0 \leq u < \bmod(v+w,p)\}).
\end{align}
Since $A_{k,h} \cup \{p(k\oplus h)+u: 0 \leq u < \bmod(v+w,p)\} \subset A_{k,h} \cup C_{k,h,v,w}$,
by (\ref{grundyequaltoAC}), (\ref{notbelongCA}), (\ref{lastequation}), and Lemma \ref{lemmaformex}.
\begin{align}
&p(k\oplus h)+ \bmod(v+w,p) \nonumber\\
& ={\rm mex}(A_{k,h} \cup \{p(k\oplus h)+u: 0 \leq u < \bmod(v+w,p)\}) \nonumber\\
&= {\rm mex}(A_{k,h} \cup C_{k,h,v,w}) = \mathcal{G}((x,y)).\nonumber 
\end{align}
\end{proof}

\section{Relation Between Grundy Number and ${\bf Move}$}

By Theorem \ref{theofsufficientcondition}, the Grundy number of the generalized Ry\={u}\={o} Nim for $p$ is 
\begin{align}\label{grundyofgenep}
\mathcal{G}((x,y)) = \bmod(x+y,p) + p(\lfloor \frac{x}{p} \rfloor	\oplus \lfloor \frac{y}{p}\rfloor).
\end{align}
In this section, we consider a necessary condition for a Nim with a new chess piece to obtain the Grundy number expressed by
(\ref{grundyofgenep}).

We define the move set in Definition \ref{defofmoveset}, and present some examples of move sets in Example \ref{exampleofmoveset}.
Althought $move((x,y))$ depends on $(x,y)$, the move set depends only on the chess piece.

\begin{defn}\label{defofmoveset}
Let $move((x,y))$ be the set of positions that can be reached by a chess piece from the position $(x,y)$.
Then, a subset $M$ of $ \mathbb{Z}_{\geq0} \times \mathbb{Z}_{\geq0}$ is said to be the move set of the chess piece when the following hold:\\
$(i)$ $(0,0) \notin M$.\\
$(ii)$ $move((x,y))$ $= \{(x-s,y-t):(s,t) \in M, s \leq x \textit{ and } t \leq y\}$.
\end{defn}

\begin{rem}
Since $(0,0) \notin M$, the move of a chess piece reduces at least one of the coordinates.
\end{rem}

\begin{exam}\label{exampleofmoveset}
$(a)$ The move set of a rook on a chess board is $M=\{ (s,0):s \in \mathcal{N}\} \cup \{ (0,t):t \in \mathcal{N}\}$, and 
$move((x,y)) = \{ (u,y):0 \leq u < x \textit{ and } u \in \mathbb{Z}_{\geq 0}\} \cup \{ (x,v):0 \leq v < y \textit{ and } v \in \mathbb{Z}_{\geq 0}\} = \{(x-s,y-t):(s,t) \in M, s \leq x \textit{ and } t \leq y\}$.\\
$(b)$ The move set of a queen on a chess board is $M=\{ (s,0):s \in \mathcal{N}\} \cup \{ (0,t):t \in \mathcal{N}\} \cup \{ (r,r):r \in \mathcal{N}\}$, and 
$move((x,y)) = \{ (u,y):0 \leq u < x \textit{ and } u \in \mathbb{Z}_{\geq 0}\} \cup \{ (x,v):0 \leq v < y \textit{ and } v \in \mathbb{Z}_{\geq 0}\} \cup \{ (x-r,y-r):1 \leq r \leq x,y \textit{ and } r \in \mathbb{Z}_{\geq 0}\} = \{(x-s,y-t):(s,t) \in M, s \leq x \textit{ and } t \leq y\}$.\\
$(c)$ The move set of the generalized Ry\={u}\={o} Nim for $p$ is \\
$ \{(s,0):s \in \mathcal{N} \} \cup \{(0,t):t \in \mathcal{N} \} \cup \{(s,t):s,t \in \mathbb{Z}_{\geq0}, 1 \leq s+t \leq p-1 \}. $
$move((x,y))= \{(u,y):u<x \textit{ and } u \in \mathbb{Z}_{\geq 0}\} \cup \{(x,v):v<y \textit{ and } v \in \mathbb{Z}_{\geq 0}\}
\cup \{(x-s,y-t):1 \leq s \leq x, 1 \leq t \leq y, s,t \in \mathbb{Z}_{\geq 0} \textit{ and } s+t \leq p-1\} = \{(x-s,y-t):(s,t) \in M, s \leq x \textit{ and } t \leq y\}$.
\end{exam}

Theorem \ref{theofnecessarycondition} shows that 
 a new chess game has to have the move set that include the move set of the generalized Ry\={u}\={o} Nim for $p$ to have 
 the Grundy number $\mathcal{G}^{ \prime} ((x,y)) = \bmod(x+y,p) + p(\lfloor \frac{x}{p} \rfloor	\oplus \lfloor \frac{y}{p}\rfloor).$

\begin{thm}\label{theofnecessarycondition}
Suppose that we make a variant of Wythoff Nim using a new chess piece with the restriction that, by moving, no coordinate increases, and at least one of the coordinates reduces. We also suppose that the Grundy number of this game satisfies Equation (\ref{grundygyakuform}).

\begin{align}\label{grundygyakuform}
\mathcal{G}^{ \prime} ((x,y)) = \bmod(x+y,p) + p(\lfloor \frac{x}{p} \rfloor	\oplus \lfloor \frac{y}{p}\rfloor).
\end{align}
Then, the move set $M$ of this new chess piece satisfies (\ref{moveset1}).
\begin{align}\label{moveset1}
M &\supset \{(s,0):s \in \mathcal{N} \} \cup \{(0,t):t \in \mathcal{N} \}\nonumber\\
&\cup \{(s,t):s,t \in \mathbb{Z}_{\geq0}, 1 \leq s+t \leq p-1 \}.
\end{align}
\end{thm}

\begin{proof}
By (\ref{grundygyakuform}), we have
\begin{align}
& \mathcal{G}^{ \prime}((0,0))=0, \nonumber \\
& \mathcal{G}^{ \prime}((1,0)) =1, \mathcal{G}^{ \prime}((0,1)) =1, \nonumber \\
& \mathcal{G}^{ \prime}((2,0)) =2, \mathcal{G}^{ \prime}((1,1)) =2, \mathcal{G}^{ \prime}((0,2)) =2, \nonumber \\
& \cdots \nonumber \\
& \mathcal{G}^{ \prime}((p-1,0)) =p-1, \mathcal{G}^{ \prime}((p-2,1)) =p-1,\cdots \nonumber\\
& \textit{ and } \mathcal{G}^{ \prime}((0,p-1)) =p-1.\label{listofgrund}
\end{align}

For any $s,t$ such that $0 < s+t \leq p-1$,
$\mathcal{G}^{ \prime}((s,t)) =s+t > 0$, and hence, by the definition of the Grundy number, $move((s,t))$ contains a position whose Grundy number is $0$.
In this game, no coordinate increases when players move the chess piece; thus, 
$move((s,t)) \subset \{(u,v):u+v < s+t \leq p-1 \}$. By (\ref{listofgrund}), $(0,0)$ is the only position in $move((s,t))$ whose Grundy number is $0$.

Therefore,
\begin{align}\label{movepart1}
&(0,0) \in move((s,t)) \textit{ for any } s,t \in \mathbb{Z}_{\geq 0} \textit{ such that } 1 \leq s+t \leq p-1.
\end{align}	

By (\ref{movepart1}), 
$M \supset \{(s,t):s,t \in \mathbb{Z}_{\geq0}, 1 \leq s+t \leq p-1 \}$.
Let $s$ be an arbitrary non-negative integer. Then, there exist $s^{\prime}, u \in \mathbb{Z}_{\geq 0}$ such that $s = ps^{\prime} + u$ and $0 \leq u < p$.
Then, $\mathcal{G}^{ \prime}((s,0)) = \bmod(ps^{\prime}+u,p) + p(\lfloor \frac{ps^{\prime}+u}{p} \rfloor	\oplus \lfloor \frac{0}{p}\rfloor)$
$=u+p\lfloor \frac{s}{p} \rfloor = u +ps^{\prime} = s$, and hence,
$\mathcal{G}^{ \prime}((s,0)) = s > 0$ for any natural number $s$.
By the definition of the Grundy number, $move((s,0))$ contains a position whose Grundy number is $0$.
Therefore, 
\begin{align}\label{movepart2}
(0,0) \in move((s,0)) \textit{ for any } s \in \mathbb{Z}_{\geq 0} \textit{ such that } 0 < s.
\end{align}
By (\ref{movepart2}),
$M \supset \{(s,0):s \in \mathcal{N} \}$.
Similarly,
\begin{align}
&(0,0) \in move((0,t)) \textit{ for any } t \in \mathbb{Z}_{\geq 0} \textit{ such that } 0 < t ,\label{movepart3}\\
&\textit{ and } M \supset \{(0,t):t \in \mathcal{N} \}. \nonumber
\end{align}
Therefore, we have (\ref{moveset1}).
\end{proof}

\section{Generalized {\boldmath Ry\={u}\={o} } Nim with a Pass Move}
An interesting, albeit extremely difficult, problem in combinatorial game theory is to determine what happens when standard game rules are modified so as to allow for a one-time pass, i.e., a pass move which may be used at most once in a game, and not from a terminal position. Once the pass move has been used by either player, it is no longer available.
In this section, we study a generalized Ry\={u}\={o} Nim with a pass move. 
In a generalized Ry\={u}\={o} Nim with a pass move, a position is represented by three coordinates $(x,y,pass)$, where $pass$ denotes the pass move.
When $pass = 1$ or $pass = 0$, then a pass move is available or not available, respectively.
Since a pass move may not be used from a terminal position, a player cannot move from $(0,0,1)$ to $(0,0,0)$ by using a pass move. 
Although there does not seem to be a simple formula for Grundy numbers of Ry\={u}\={o} Nim with a pass move, there are 
simple formulas for the $\mathcal{P}$-position of Ry\={u}\={o} Nim and the generalized Ry\={u}\={o} Nim for $p$ with a pass move.
\begin{defn}\label{defofmovepass}
We define the move of the generalized Ry\={u}\={o} Nim for $p$ with a pass move. 
Let $x,y \in \mathbb{Z}_{\ge 0}$ such that $1 \leq x+y$. We have two cases.\\
$(a)$ 
We define the move of the generalized Ry\={u}\={o} Nim for $p$ when a pass move is not available. 

\noindent
Let 
\begin{align}
& M_{pass1}= \{(u,y,0):u\in \mathbb{Z}_{\geq 0} \textit{ and } u<x\}, \label{ryuuoumoveex1} \\
& M_{pass2}= \{(x,v,0):v\in \mathbb{Z}_{\geq 0} \textit{ and } v<y\} \label{ryuuoumoveex2} \\
& \textit{ and } \nonumber \\
& M_{pass3}= \{(x-s,y-t,0):s,t \in \mathbb{Z}_{\geq 0}, 1 \leq s \leq x, 1\leq t \leq y\nonumber\\
& \textit{ and } s+t \leq p-1\}. \label{ryuuoumoveex3}
\end{align}

We define 
\begin{align}
move(x,y,0)	=M_{pass1} \cup M_{pass2} \cup M_{pass3}. 
\end{align}
$(b)$ 
We define the move of the generalized Ry\={u}\={o} Nim for $p$ when a pass move is available. \\
Let
\begin{align}
& M_{pass4}= \{(u,y,1):u\in \mathbb{Z}_{\geq 0} \textit{ and } u<x\}, \label{ryuuoumoveex1p} \\
& M_{pass5}= \{(x,v,1):v\in \mathbb{Z}_{\geq 0} \textit{ and } v<y\}, \label{ryuuoumoveex2p} \\
& M_{pass6}= \{(x-s,y-t,1):s,t \in \mathbb{Z}_{\geq 0}, \nonumber\\
& 1 \leq s \leq x, 1\leq t \leq y \textit{ and } s+t \leq p-1\} \label{ryuuoumoveex3p}\\
& \textit{ and } \nonumber \\		
& M_{pass7} = \{(x,y,0)\}. \label{ryuuoupassmove}
\end{align}	

We define 
\begin{align}
move(x,y,1)	=M_{pass4} \cup M_{pass5} \cup M_{pass6} \cup M_{pass7}. 
\end{align}
\end{defn}

\begin{rem}
In Figure \ref{geryuuoumovepg}, the sets (\ref{ryuuoumoveex1}) and (\ref{ryuuoumoveex1p}), (\ref{ryuuoumoveex2}) and (\ref{ryuuoumoveex2p}), and (\ref{ryuuoumoveex3}) and (\ref{ryuuoumoveex3p}) denote the horizontal, vertical, and upper left moves, and the set (\ref{ryuuoupassmove}) denotes a pass move.
$M_{pass1}, M_{pass2}$, and $M_{pass3}$ depend on $(x,y,0)$, and $M_{pass4}, M_{pass5},M_{pass6}$, and $M_{pass7}$ depend on $(x,y,1)$. Hence, it is more precise to write $M_{pass1}(x,y,0)$ $,\cdots, $ $M_{pass7}(x,y,1)$. If we write $M_{pass1}(x,y,0)$ instead of $M_{pass1}$, the relations and equations appear more complicated. Therefore,
we omit $(x,y,0)$ and $(x,y,1)$ for convenience. 
\end{rem}

\begin{exam}
When $pass = 0$, i.e., a pass move is not available, the game becomes the same as the game we studied in Section \ref{generalnim}.
Table \ref{ryuuoupassgrundy} lists the Grundy numbers $\mathcal{G}(x,y,1)$ for $x,y = 0,1,2,...,12$ of the general Ry\={u}\={o} Nim for $p=3$ with a pass move. 
Since the third coordinate $pass =1$, these are the Grundy numbers of positions with an available pass move.
It seems that there is no simple formula for the Grundy numbers of the general Ry\={u}\={o} Nim for $p=3$ with a pass move in Table \ref{ryuuoupassgrundy}.

\begin{table}[!htb]
\begin{center}
\caption{Grundy numbers $\mathcal{G}(x,y,1)$ for $x,y = 0,1,2,\cdots ,12$ 
of the general Ry\={u}\={o} Nim for $p=3$ when a pass move is available.}
\label{ryuuoupassgrundy}
\includegraphics[height=0.45\columnwidth,bb=0 0 398 275]{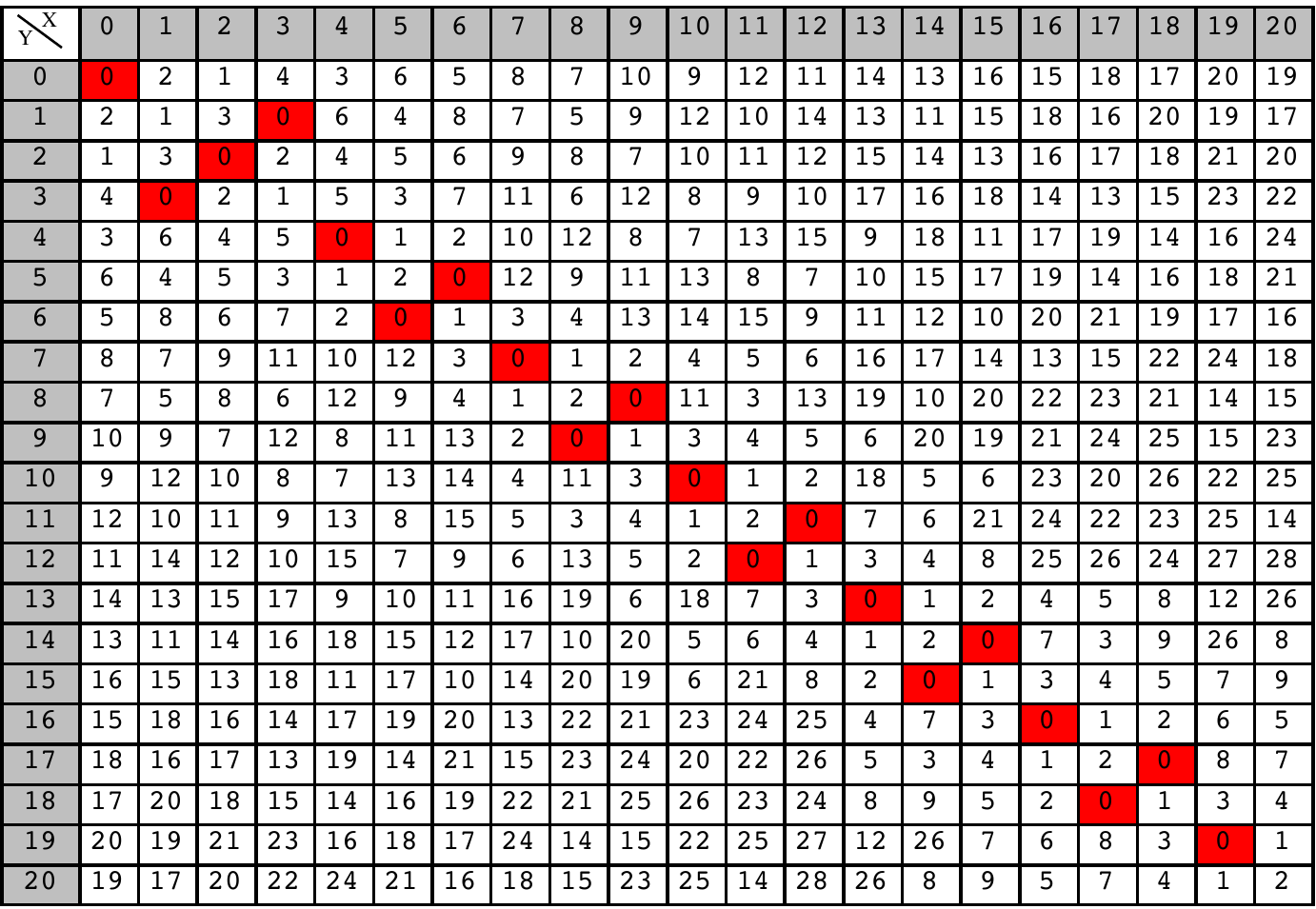}
\end{center}
\end{table}

The Grundy number of positions in red squares is zero, and hence, they are $\mathcal{P}$-positions (these positions are in dark gray when this article is printed in black and white.) There seems to be a certain pattern 
in the set of $\mathcal{P}$-positions, and clearly, these $\mathcal{P}$-positions can be divided into three groups of positions.
The first group consists of three adjoining $\mathcal{P}$-positions, and Table \ref{tableofp11} shows this group. We denote this by $\mathcal{P}_{1,1}$.

\begin{table}[!htb]
\begin{center}
\caption{The set $\mathcal{P}_{1,1}$ for $p=3$}
\label{tableofp11}
\includegraphics[height=0.45\columnwidth,bb=0 0 398 275]{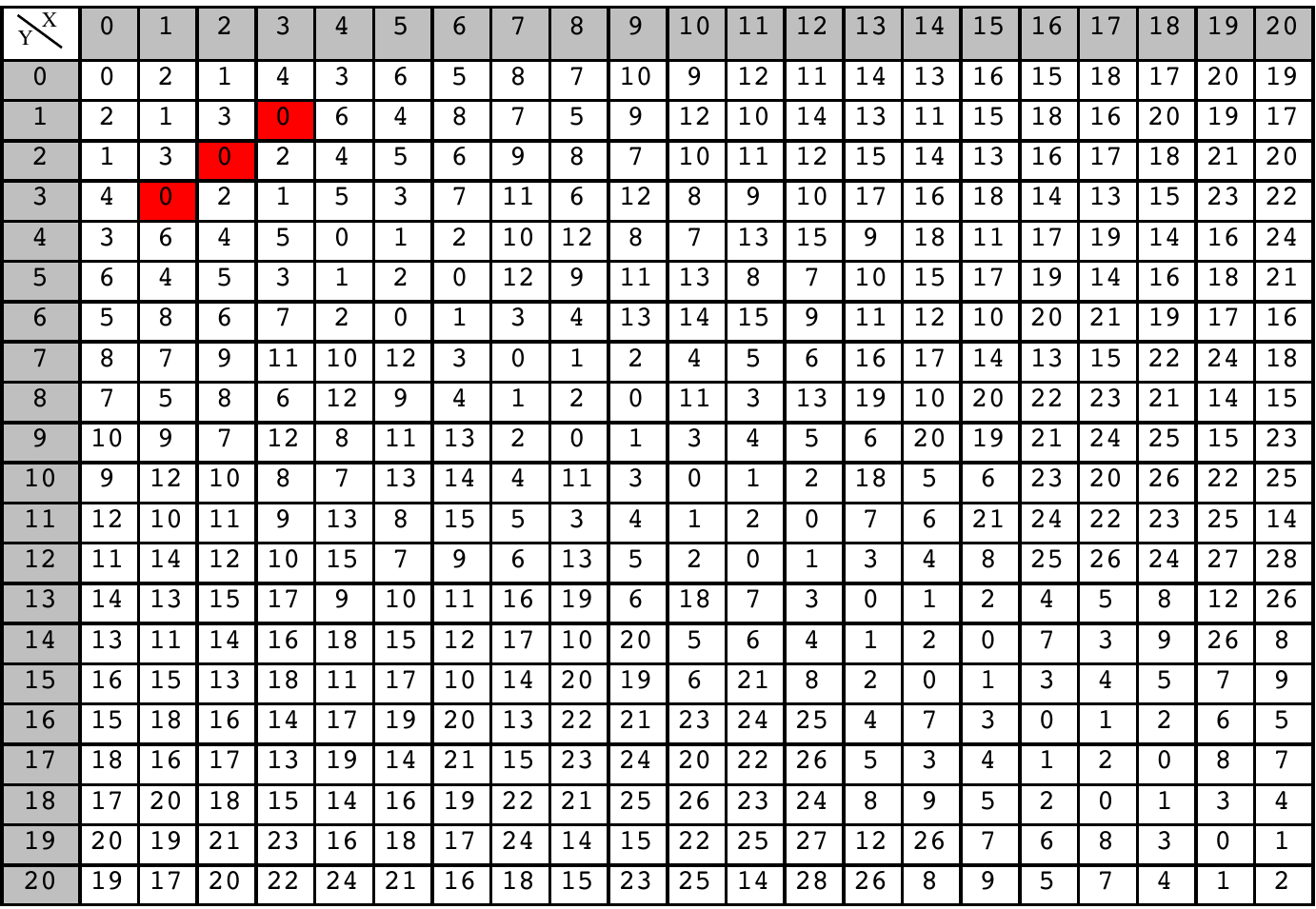}
\end{center}
\end{table}

The second group consists of positions that appear diagonally with a certain interval, and Table \ref{tableofp12} summarizes this group. We denote this group by $\mathcal{P}_{1,2}$.

\begin{table}[!htb]
\begin{center}
\caption{The set $\mathcal{P}_{1,2}$ for $p=3$}
\label{tableofp12}
\includegraphics[height=0.45\columnwidth,bb=0 0 398 275]{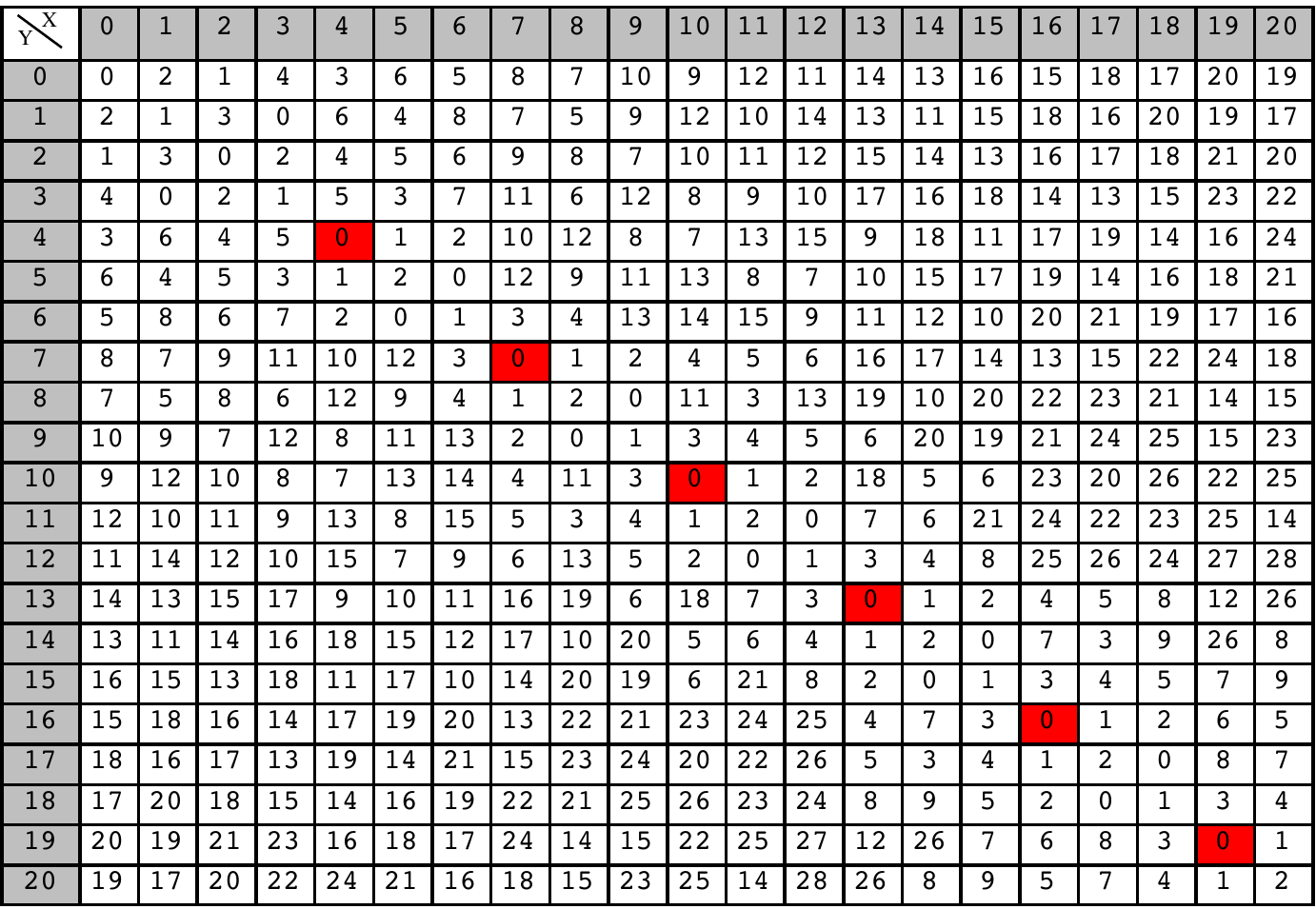}
\end{center}
\end{table}

The third group consists of smaller groups that appear with a certain interval, and each group has two adjoining positions. Table \ref{tableofp13} summarizes this group. We denote this group by $\mathcal{P}_{1,3}$.

\begin{table}[!htb]
\begin{center}
\caption{The set $\mathcal{P}_{1,3}$ for $p=3$}
\label{tableofp13}
\includegraphics[height=0.45\columnwidth,bb=0 0 398 275]{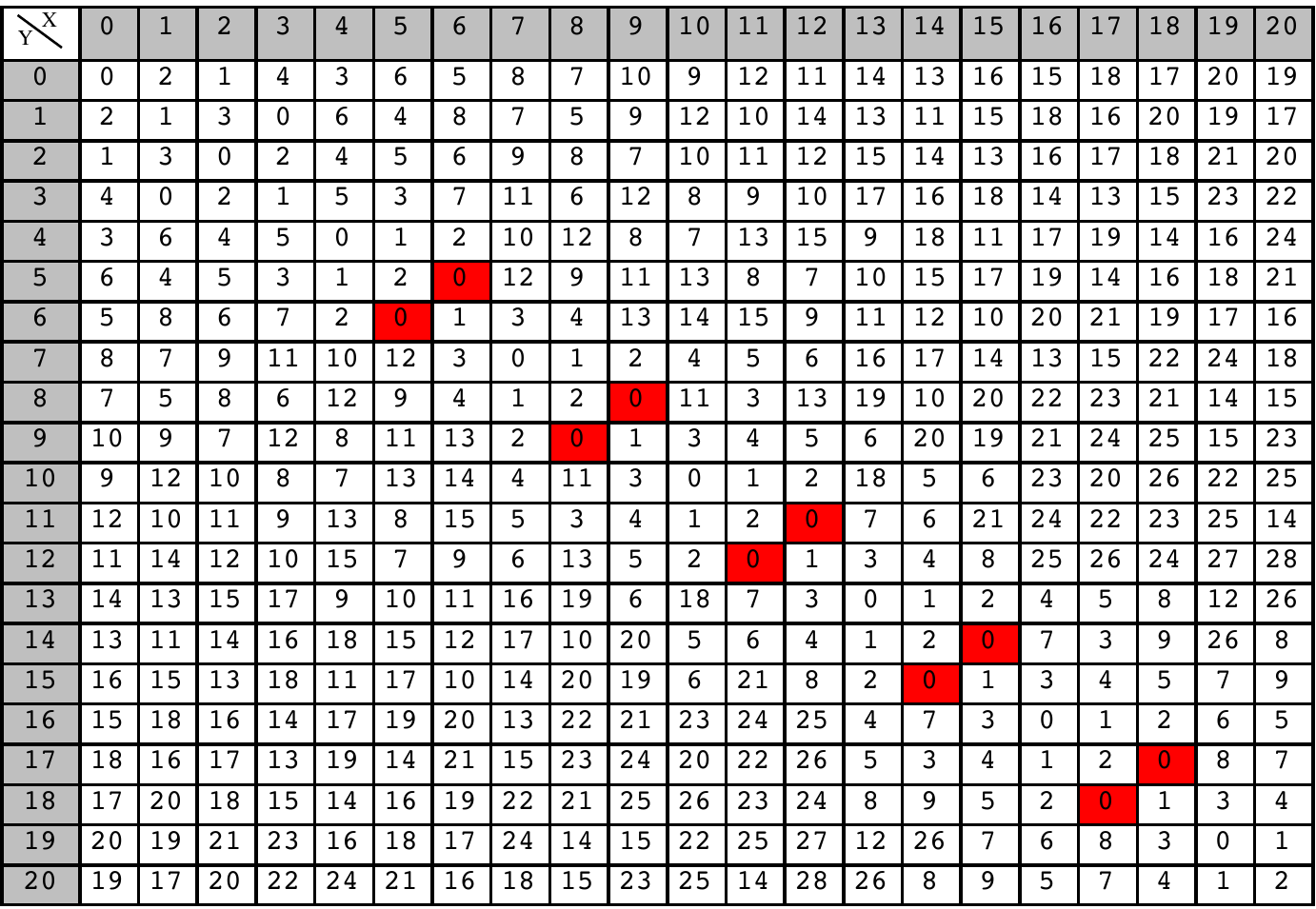}
\end{center}
\end{table}
\end{exam}

In Definition \ref{defofP0P1sets}, we define sets $\mathcal{P}$ and $\mathcal{N}$, and in Theorem \ref{PNsetsforpassmove}, we prove that $\mathcal{P}$ and $\mathcal{N}$ are the sets of $\mathcal{P}$-positions and $\mathcal{N}$-positions of the generalized Ry\={u}\={o} Nim for $p$ with a pass move, respectively.
The set $\mathcal{P}$ consists of four sets: $\mathcal{P}_0$, $\mathcal{P}_{1,1}$, $\mathcal{P}_{1,2}$, and $\mathcal{P}_{1,3}$.

\begin{defn}\label{defofP0P1sets}
Let $\mathcal{P}_0 = \{(x,y,0): x+y = 0 \ (\bmod \ p) \textit{ and } \lfloor \frac{x}{p}\rfloor = \lfloor \frac{y}{p}\rfloor \},$
$\mathcal{P}_{1,1} = \{(m+1,p-m,1): m \in \mathbb{Z}_{\geq 0} \textit{ and } 0 \leq m \leq p-1 \}$, $\mathcal{P}_{1,2} = \{(pn+1,pn+1,1):n \in \mathcal{N}\}$, and 
$\mathcal{P}_{1,3} = \{(k+pn,p+2-k+pn,1):n \in \mathcal{N}, 2 \leq k \leq p \textit{ and } k \in \mathcal{N} \}$.
Let $\mathcal{P}_1 = \mathcal{P}_{1,1} \cup \mathcal{P}_{1,2} \cup \mathcal{P}_{1,3}$ and 
$\mathcal{P} = \mathcal{P}_0 \cup \mathcal{P}_1$.\\
Let $\mathcal{N}_0 = \{(x,y,0):x,y \in \mathbb{Z}_{\geq 0}\}-\mathcal{P}_0$, $\mathcal{N}_1 = \{(x,y,1):x,y \in \mathbb{Z}_{\geq 0}\}-\mathcal{P}_1$ and $\mathcal{N} = \mathcal{N}_0 \cup \mathcal{N}_1$.
\end{defn}
\begin{rem}
Tables \ref{tableofp11}, \ref{tableofp12}, and \ref{tableofp13} summarize the following groups of positions: 
$\mathcal{P}_{1,1}$, $\mathcal{P}_{1,2}$, and $\mathcal{P}_{1,3}$ for $p=3$. 
\end{rem}
\begin{exam}
Table \ref{ryuuoupassgrundyp4} lists the Grundy numbers of the generalized Ry\={u}\={o} Nim
for $p=4$. Tables \ref{tableofp11p4}, \ref{tableofp12p4}, and \ref{tableofp13p4} summarize the sets 
$\mathcal{P}_{1,1}$, $\mathcal{P}_{1,2}$, and $\mathcal{P}_{1,3}$
for $p=4$. 
\begin{table}[!htb]
\begin{center}
\caption{Grundy numbers $\mathcal{G}(x,y,1)$ for $x,y = 0,1,2,\cdots ,12$ 
of the general Ry\={u}\={o} Nim for $p=4$ when a pass move is available.}
\label{ryuuoupassgrundyp4}
\includegraphics[height=0.45\columnwidth,bb=0 0 398 275]{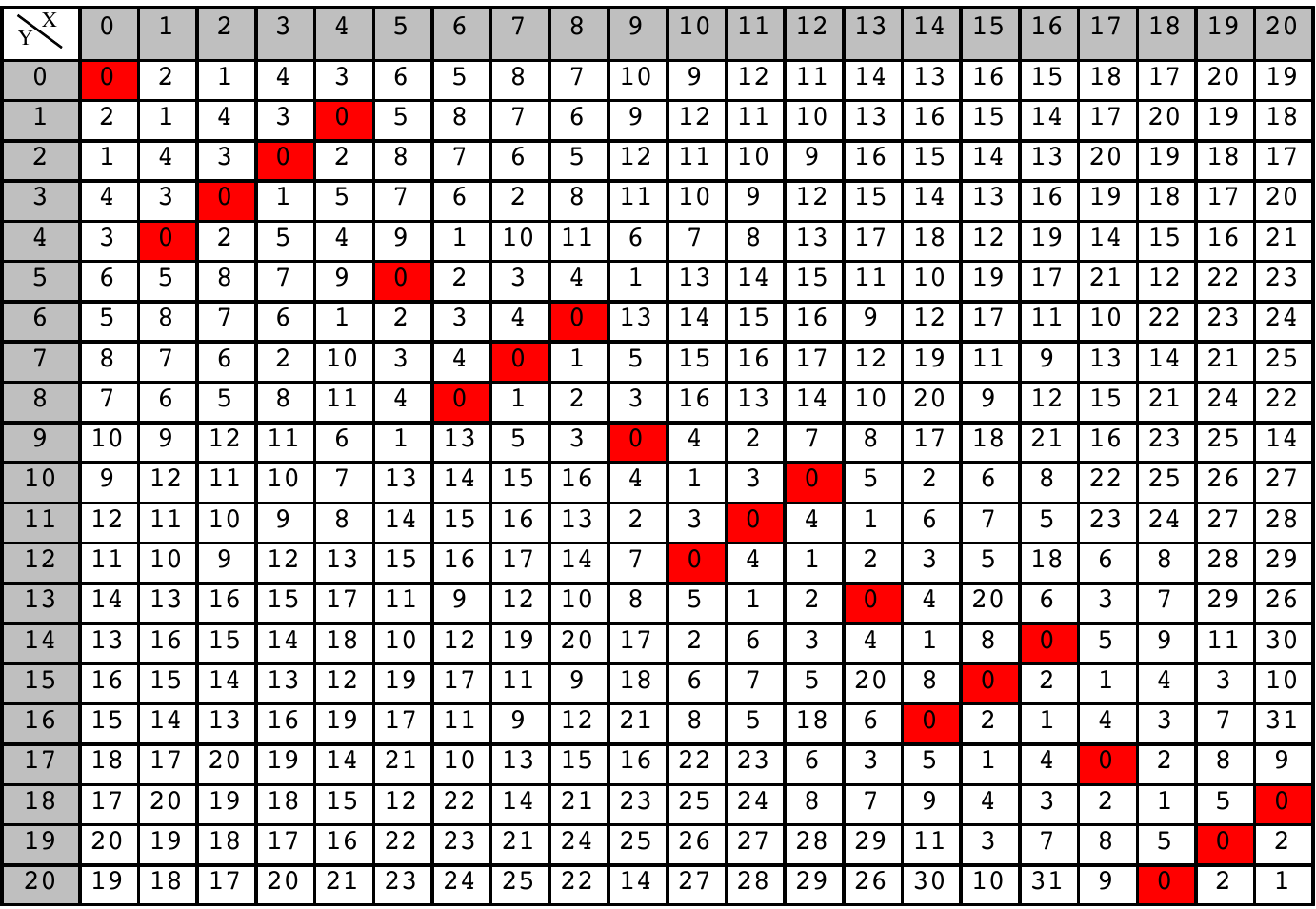}
\end{center}
\end{table}
\begin{table}[!htb]
\begin{center}
\caption{The set $\mathcal{P}_{1,1}$ for $p=4$}
\label{tableofp11p4}
\includegraphics[height=0.45\columnwidth,bb=0 0 398 275]{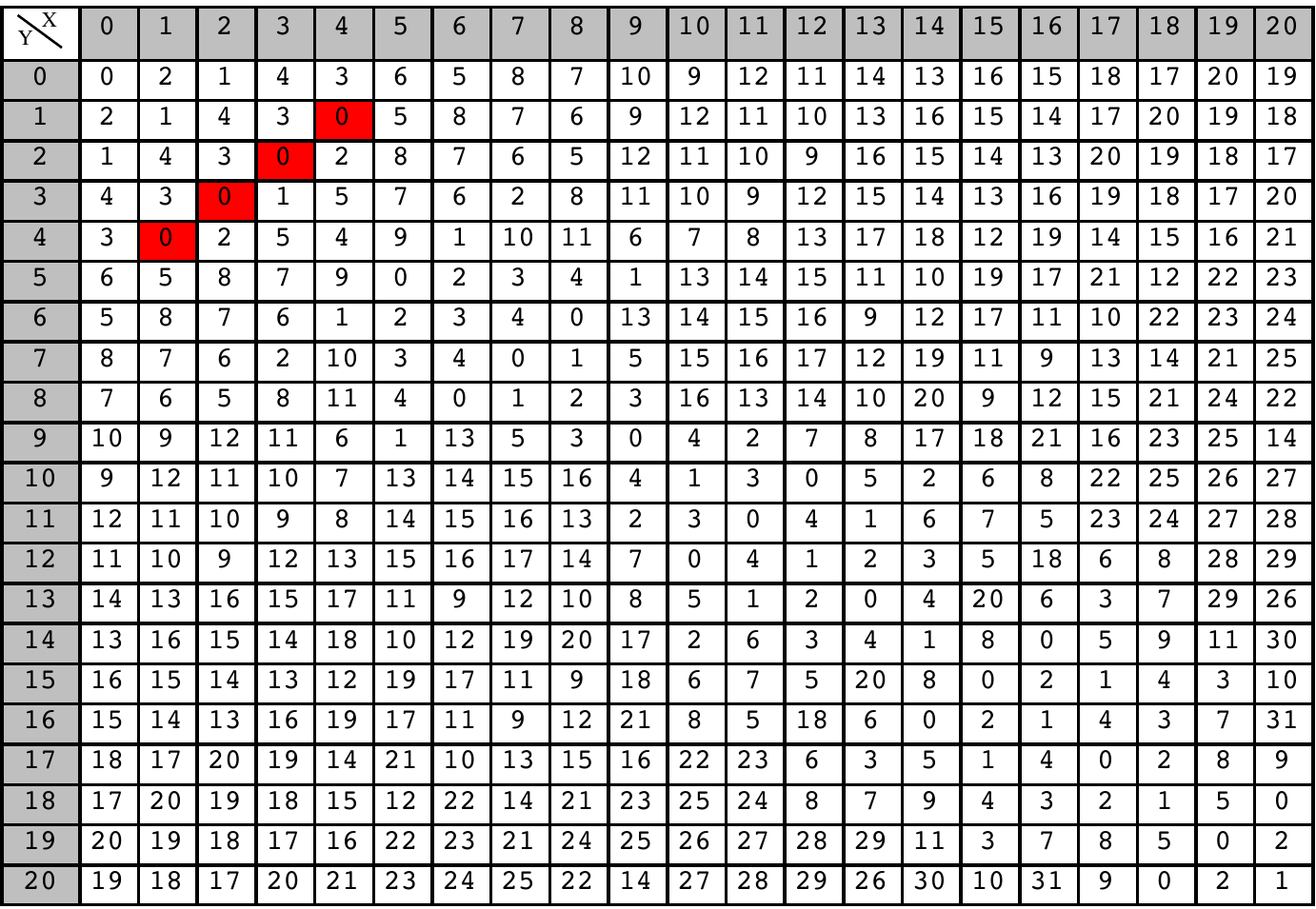}
\end{center}
\end{table}
\begin{table}[!htb]
\begin{center}
\caption{The set $\mathcal{P}_{1,2}$ for $p=4$}
\label{tableofp12p4}
\includegraphics[height=0.45\columnwidth,bb=0 0 398 275]{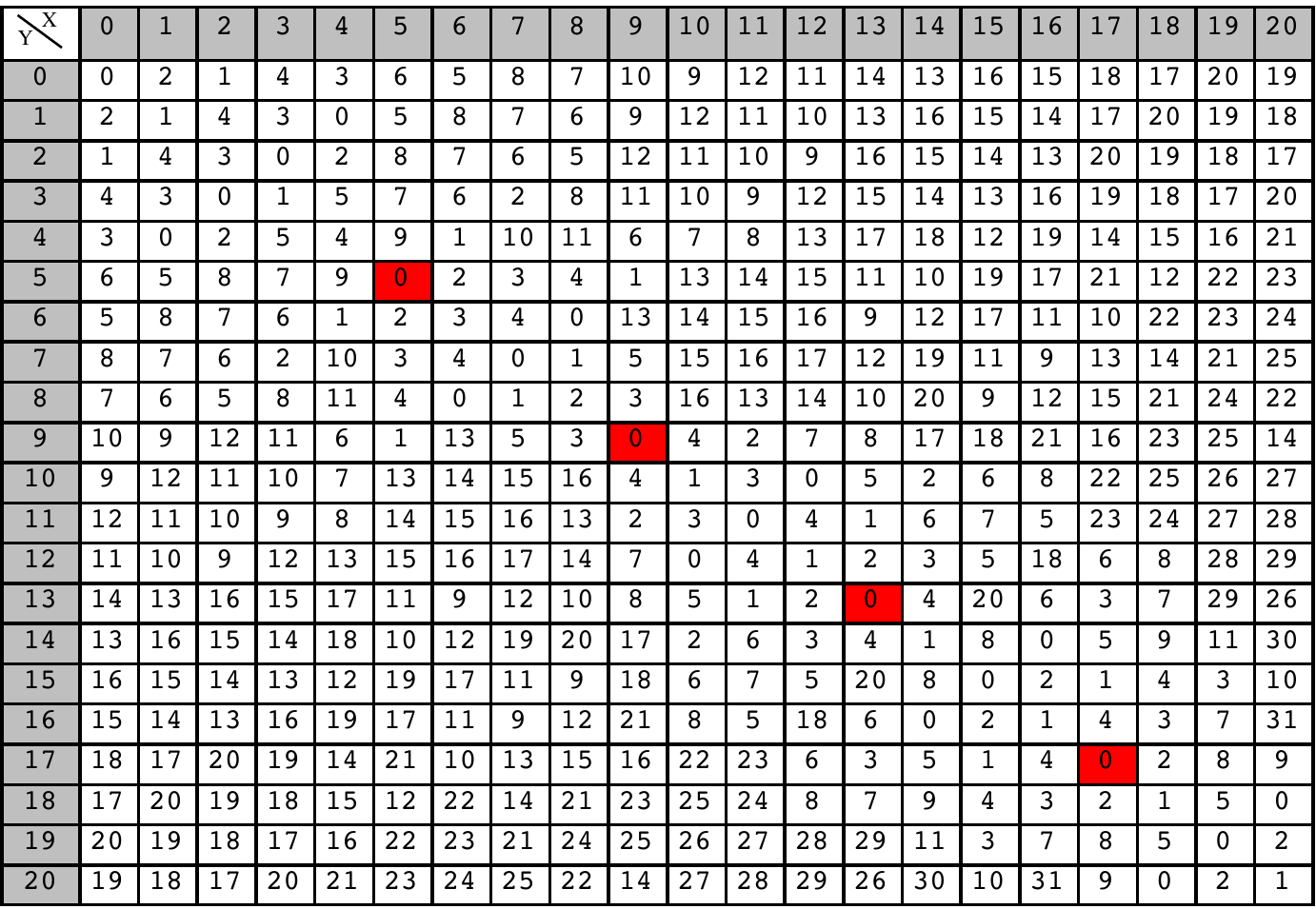}
\end{center}
\end{table}
\begin{table}[!htb]
\begin{center}
\caption{The set $\mathcal{P}_{1,3}$ for $p=4$}
\label{tableofp13p4}
\includegraphics[height=0.45\columnwidth,bb=0 0 398 275]{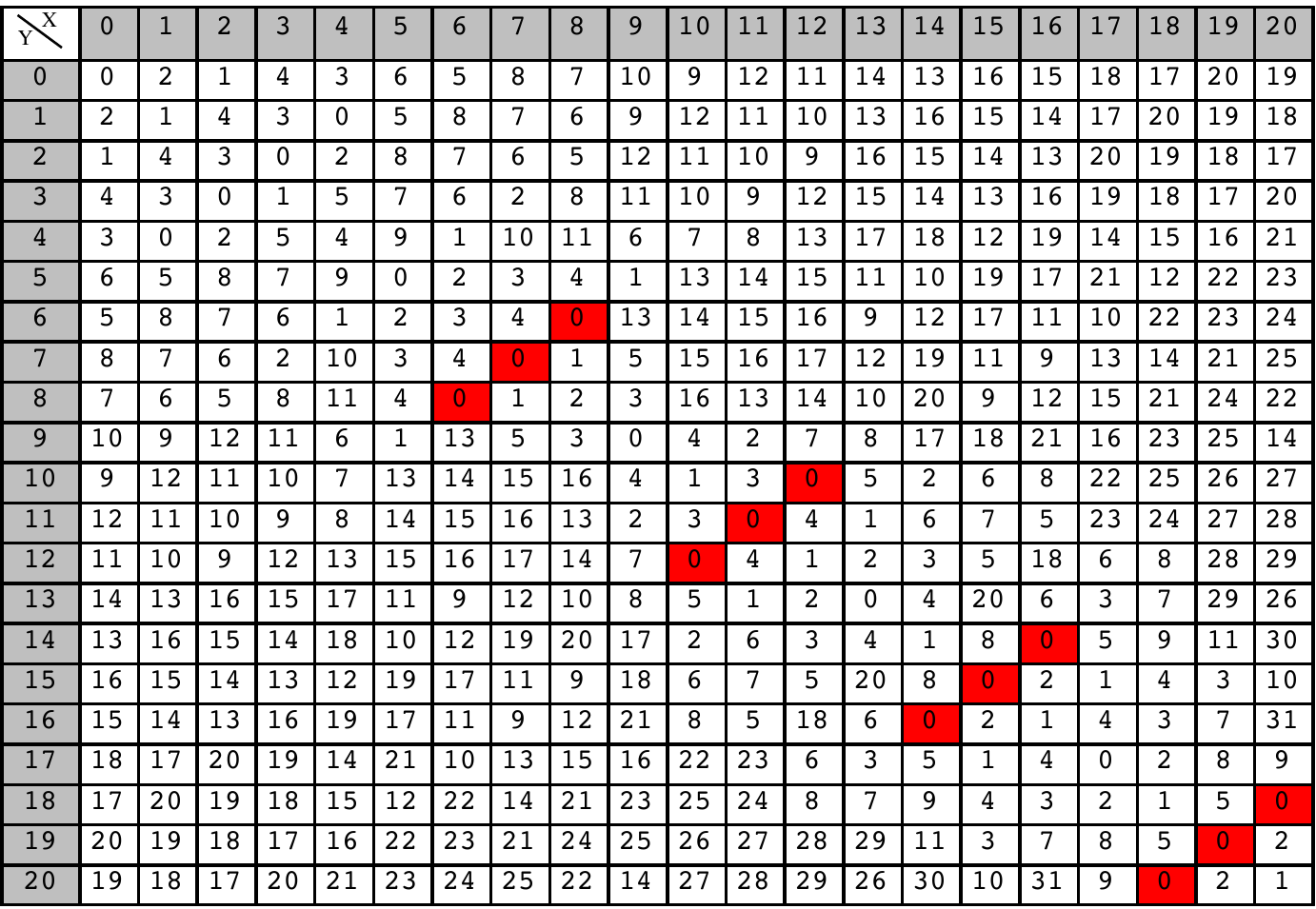}
\end{center}
\end{table}
\end{exam}

\begin{Lem}\label{ppositionp0}
$\mathcal{P}_0$ and $\mathcal{N}_0$ are the set of $\mathcal{P}$-positions and the set of $\mathcal{N}$-positions when a pass move is not available, respectively.
\end{Lem}

\begin{proof}
Since $pass = 0$, by Theorem \ref{theoremofsumg}, a position $(x,y,0)$ of a generalized Ry\={u}\={o} Nim for $p$ is a $\mathcal{P}$-position if and only if
\begin{align*}\label{gxyequql0}
\mathcal{G}((x,y)) = \bmod(x+y,p) + p(\lfloor \frac{x}{p} \rfloor	\oplus \lfloor \frac{y}{p}\rfloor) =0. 
\end{align*}
By the definition of $\mathcal{P}_0$, (\ref{gxyequql0}) is true 
if and only if $(x,y,0) \in \mathcal{P}_0$. \\
Since $\mathcal{P}_0$ is the set of $\mathcal{P}$-positions, $\mathcal{N}_0$ is the set of $\mathcal{N}$-positions.
\end{proof}
We need three lemmas to prove Theorem \ref{PNsetsforpassmove}.
\begin{Lem}\label{lemmaforabc}
$(a)$ If $(x,y,1) \in \mathcal{P}_1$, then $x+y \geq p+1$.\\
$(b)$ If $(x,y,1) \in \mathcal{P}_{1,1}$, then $x+y = p+1$, $x+y = 1 \ (\bmod \ p)$, and $1 \leq x,y \leq p$.\\
$(c)$ If $(x,y,1) \in \mathcal{P}_{1,2} \cup \mathcal{P}_{1,3}$, then $x+y = 2 \ (\bmod \ p)$ and $x-y, y-x \leq p-2$.\\
$(d)$ If $(x,y,1) \in \mathcal{P}_{1,2} $, then $x=y = 1 \ (\bmod \ p)$. \\
$(e)$ If $(x,y,1) \in \mathcal{P}_{1,3} $, then $x+y = p+2 \ (\bmod \ 2p)$. 
\end{Lem}
\begin{proof}
$(a)$, $(b)$, $(c)$, $(d)$, and $(e)$ follow directly from Definition \ref{defofP0P1sets}.
\end{proof}
\begin{Lem}\label{fromPtononP}
$move((x,y,1)) \cap \mathcal{P} = \emptyset $ for any $(x,y,1) \in \mathcal{P}_1$.
\end{Lem}
\begin{proof}
By Definition \ref{defofmovepass}
\begin{align}
move(x,y,1)	=M_{pass4} \cup M_{pass5} \cup M_{pass6} \cup M_{pass7},
\end{align}
where
\begin{align}
&M_{pass4}= \{(u,y,1):u \in \mathcal{N} \textit{ and } u<x\}, \label{ryuuoumoveex1p2} \\
&M_{pass5}= \{(x,v,1):v \in \mathcal{N} \textit{ and } v<y\}, \label{ryuuoumoveex2p2} \\
&M_{pass6}= \{(x-s,y-t,1): s,t \in \mathcal{N} \textit{ and } s+t \leq p-1\},\label{ryuuoumoveex3p2}\\
& \textit{ and } \nonumber \\		
& M_{pass7} = \{(x,y,0)\}.\label{ryuuoupassmove2}
\end{align}
We prove that $M_{pass4} \cup M_{pass5} \cup M_{pass6} \cup M_{pass7} \cap \mathcal{P} = \emptyset$.
For any $(x,y,1) \in \mathcal{P}_1$, by $(b)$ and $(c)$ of Lemma \ref{lemmaforabc}, we have $x+y = 1$ or $2$ $ \ (\bmod \ p)$, and hence, $(x,y,0) \notin \mathcal{P}_0$. 
Therefore, $M_{pass7} \cap \mathcal{P} = \emptyset$, i.e., one cannot reach an element of $\mathcal{P}$ by using a pass move. Since $\mathcal{P}_1$ is unchanged when we switch $x$ and $y$, $M_{pass4} \cap \mathcal{P}_1 = \emptyset$ implies that $M_{pass5} \cap \mathcal{P}_1 = \emptyset$.
Therefore,
we only have to prove that 
$(M_{pass4} \cup M_{pass6} )\cap \mathcal{P}_1 = \emptyset$. We consider three cases.\\
\underline{Case $(a)$} Suppose that $(x,y,1) = (1+m,p-m,1) \in \mathcal{P}_{1,1}$ for a non-negative integer $m$ such that $0 \leq m \leq p-1$. We consider subcases $(a.1)$ and $(a.2)$.\\
\underline{Subcase $(a.1)$} Let $(m+1-u,p-m,1)$ be an arbitrary element of $M_{pass4}$ such that $u \in \mathcal{N}$ and $ u \leq m+1$.
Since $(m+1-u)+(p-m)=p+1-u < p+1$, by $(a)$ of Lemma \ref{lemmaforabc}, 
$(m+1-u,p-m,1-v) \notin \mathcal{P}_1$. Therefore, $M_{pass4} \cap \mathcal{P}_1 = \emptyset$, i.e.,
from $(m+1,p-m,1)$, one cannot reach an element in $\mathcal{P}_1$ by reducing the first coordinate.\\
\underline{Subcase$(a.2)$} Let $(m+1-u,p-m-v,1)$ be an arbitrary element of $M_{pass6}$ such that $s,t \in \mathcal{N} \textit{ and } s+t \leq p-1.$
Since $ (m+1-u) + (p-m-v) \leq p$, by $(a)$ of Lemma \ref{lemmaforabc}, $M_{pass6} \cap \mathcal{P}_1 = \emptyset$.\\
By subcases $(a.1)$ and $(a.2)$, for any $(x,y,1) \in \mathcal{P}_{1,1}$,
\begin{align}\label{resultofcasea}
(M_{pass4} \cup M_{pass6} )\cap \mathcal{P}_1 = \emptyset.		
\end{align}
\underline{Case $(b)$} Let $(1+pn,1+pn,1) \in \mathcal{P}_{1,2}$. We consider subcases $(b.1)$ and $(b.2)$.\\
\underline{Subcase$(b.1)$} Let $(1+pn-u,1+pn,1)$ be an arbitrary element of $M_{pass4}$.
Since $1+pn >p$, the second coordinate is greater than $p$. Therefore, by $(b)$ of Lemma \ref{lemmaforabc}, $(1+pn-u,1+pn,1) \notin \mathcal{P}_{1,1}.$
Since $1+pn-u < 1+pn$, $(1+pn-u,1+pn,1) \notin \mathcal{P}_{1,2}.$\\
There exists no $u \in \mathcal{N}$ such that 
$(1+pn-u)+(1+pn) = 2 \ (\bmod \ p)$ and $(1+pn)-(1+pn-u) \leq p-2$.
Therefore, by $(c)$ of Lemma \ref{lemmaforabc},
$(1+pn-u,1+pn,1) \notin \mathcal{P}_{1,3}.$
Therefore, $M_{pass4} \cap \mathcal{P}_1 = \emptyset.$\\
\underline{Subcase$(b.2)$} Let $(1+pn-u,1+pn-v,1)$ be an arbitrary element of $M_{pass6}$ such that $s,t \in \mathcal{N} \textit{ and } s+t \leq p-1$.
Since $(1+pn-u)+(1+pn-v) \neq p+1$, by $(b)$ of Lemma \ref{lemmaforabc}, $(1+pn-u,1+pn-v) \notin \mathcal{P}_{1,1}$.\\
Since $2np+2 > (1+pn-u)+(1+pn-v) \geq (2n-1)p + 3$, 
$(pn+1-u)+(pn+1-v) \neq 2 \ (\bmod \ p)$. Therefore, by $(c)$ of Lemma \ref{lemmaforabc}, 
$(pn+1-u,pn+1-v) \notin \mathcal{P}_{1,2} \cup \mathcal{P}_{1,3}$, and $M_{pass6} \cap \mathcal{P}_1 = \emptyset.$\\
By subcases $(b.1)$ and $(b.2)$ for any $(x,y,1) \in \mathcal{P}_{1,2}$

\begin{equation}\label{resultofcaseb}
(M_{pass4} \cup M_{pass6} )\cap \mathcal{P}_1 = \emptyset.
\end{equation}

\noindent
\underline{Case $(c)$} Let $(k+pn,p+2-k+pn,1) \in \mathcal{P}_{1,3}$ such that $n \in \mathcal{N}, 2 \leq k \leq p \textit{ and } k \in \mathcal{N}$. We consider subcases $(c.1)$ and $(c.2)$\\
\underline{Subcase$(c.1)$} Let $(k+pn-u,p+2-k+pn,1)$ be an arbitrary element of $M_{pass4}$.\\
Since $p+2-k+pn > p$, by $(b)$ of Lemma \ref{lemmaforabc}, $(k+pn-u,p+2-k+pn,1) \notin \mathcal{P}_{1,1}.$
There exists no $k \in \mathbb{Z}_{\geq 0}$ and $u \in \mathcal{N}$ such that $2 \leq k \leq p$ and 
$k+pn-u = p+2-k+pn = 1$ $ \ (\bmod \ p)$. Therefore, by $(d)$ of Lemma \ref{lemmaforabc},
$(k+pn-u,p+2-k+pn,1) \notin \mathcal{P}_{1,2}.$ 
There exists no $u$ such that $(k+pn-u) + (p+2-k+pn) = p+2 \ (\bmod \ 2p)$ and 
$ (p+2-k+pn) -(k+pn-u) \leq p-2$. Therefore, by $(c)$ and $(e)$ of Lemma \ref{lemmaforabc},
$(k+pn-u,p+2-k+pn,1) \notin \mathcal{P}_{1,3}.$ Therefore, $M_{pass4} \cap \mathcal{P}_1 = \emptyset.$\\
\underline{Subcase$(c.2)$} Let $(k+pn-u,p+2-k+pn-v,1)$ be an arbitrary element of $M_{pass6}$ such that 
$u,v \in \mathcal{N}$ and $u+v \leq p-1$.
Then, 
\begin{align}\label{p2lgkplesnp}
2pn+3 \leq (k+pn-u)+(p+2-k+pn-v) \leq (2n+1)p+1.
\end{align}
By (\ref{p2lgkplesnp}), $(k+pn-u)+(p+2-k+pn-v) > p+1$, and hence, by $(b)$ of Lemma \ref{lemmaforabc},
$(k+pn-u,p+2-k+pn-v) \notin \mathcal{P}_{1,1}$.\\	
By (\ref{p2lgkplesnp}), $(k+pn-u)+(p+2-k+pn-v) \neq 2 \ (\bmod \ p)$, and hence, by $(c)$ of Lemma \ref{lemmaforabc},
$(k+pn-u,p+2-k+pn-v) \notin \mathcal{P}_{1,2} \cup \mathcal{P}_{1,3}.$\\
Therefore, $M_{pass6} \cap \mathcal{P}_1 = \emptyset.$\\
By subcases $(c.1)$ and $(c.2)$ for any $(x,y,1) \in \mathcal{P}_{1,3}$,
\begin{equation}\label{resultofcasec}
(M_{pass4} \cup M_{pass6} )\cap \mathcal{P}_1 = \emptyset.		
\end{equation}
By (\ref{resultofcasea}), (\ref{resultofcaseb}), and (\ref{resultofcasec}),
$(M_{pass4} \cup M_{pass6} )\cap \mathcal{P}_1 = \emptyset$ for any $(x,y,1) \in \mathcal{P}_1$.
\end{proof}
\begin{Lem}\label{fromnonPtoP}
$move((x,y,1)) \cap \mathcal{P} \neq \emptyset $ for any $(x,y,1) \in \mathcal{N}_1$. 
\end{Lem}
\begin{proof}
We consider six cases.\\
\underline{Case $(a)$} Let $(x,y,1)$ = $(k+pn,h+pn^{\prime},1) \notin \mathcal{P}_1$ such that $n,n^{\prime} \in \mathcal{N}$, $n^{\prime} > n$, $k,h \in \mathbb{Z}_{\geq 0}$, and $k,h \leq p-1$. We consider subcases $(a.1)$, $(a.2)$, $(a.3)$, and $(a.4)$.\\
\underline{Subcase $(a.1)$} If $k > 2$, then $(k+pn,p+2-k+pn) \in M_{pass5} \cap \mathcal{P}_{1,3} \subset move((x,y,1)) \cap \mathcal{P}$.\\
\underline{Subcase $(a.2)$} Suppose that $k = 2$. We consider subsubcases $(a.2.1)$ and $(a.2.2)$.\\
\underline{Subsubcase $(a.2.1)$} If $h > 0$, $(k+pn,p+2-k+pn) \in M_{pass5} \cap \mathcal{P}_{1,3} \subset move((x,y,1)) \cap \mathcal{P}$.\\
\underline{Subsubcase $(a.2.2)$} Suppose that $h = 0$. If $n^{\prime} = n+1$, then $(x,y,1)$ = $(k+pn,h+pn^{\prime},1)$ $ = (pn+2,pn+p,1) \in \mathcal{P}_{1,3}$. This contradicts the assumption that $(x,y,1) \notin \mathcal{P}_1$.
Therefore, we assume that $n^{\prime} > n+1$. Then, $(k+pn,p+2-k+pn) \in M_{pass5} \cap \mathcal{P}_{1,3} \subset move((x,y,1)) \cap \mathcal{P}$.\\
\underline{Subcase $(a.3)$} If $k = 1$, then $(k+pn,2-k+pn)=(1+pn,1+pn) \in M_{pass5} \cap \mathcal{P}_{1,2} \subset move((x,y,1)) \cap \mathcal{P}$.\\
\underline{Subcase $(a.4)$} Suppose that $k = 0$. We consider subsubcases $(a.4.1)$ and $(a.4.2)$.\\
\underline{Subsubcase $(a.4.1)$} If $n \geq 2$, then $(pn,2+p(n-1)) = (p+p(n-1),p+2-p+p(n-1),1) \in M_{pass5} \cap \mathcal{P}_{1,3} \subset move((x,y,1)) \cap \mathcal{P}$.\\
\underline{Subsubcase $(a.4.2)$} If $n =1$, then $(p,1,1) \in M_{pass5} \cap \mathcal{P}_{1,1} \subset move((x,y,1)) \cap \mathcal{P}$.\\
\underline{Case $(b)$} Let $(x,y,1)$ = $(k+pn^{\prime},k+pn,1) \notin \mathcal{P}_1$ such that $n^{\prime} > n$, $k,h \in \mathbb{Z}_{\geq 0}$, and $k,h \leq p-1$. Then, we use a method that is similar to the one used in $(a)$. \\
\underline{Case $(c)$} Let 
\begin{equation}\label{conditionpnpn}
(x,y,1) = (k+pn,h+pn,1) \in \mathcal{N}_1
\end{equation}
such that $k,h \in \mathbb{Z}_{\geq 0}$ and $k,h \leq p-1$. We consider Subcase $(c.1)$, $(c.2)$ and $(c.3)$.\\
\underline{Subcase $(c.1)$} Suppose that $k=0$. We consider subsubcases $(c.1.1)$ and $(c.1.2)$.\\
\underline{Subsubcase $(c.1.1)$} If $n \geq 2$, then $(pn,2+p(n-1)) = (p+p(n-1),p+2-p+p(n-1),1) \in M_{pass5} \cap \mathcal{P}_{1,3} \subset move((x,y,1)) \cap \mathcal{P}$.\\
\underline{Subsubcase $(c.1.2)$} If $n = 1$, then $(p,1,1) \in M_{pass5} \cap \mathcal{P}_{1,1} \subset move((x,y,1)) \cap \mathcal{P}$.\\
\underline{Subcase $(c.2)$} Suppose that $h=0$. Then, we use the method that is similar to the one used in $(c.1)$.\\
\underline{Subcase $(c.3)$} Suppose that $k,h \geq 1$. We consider subsubcases $(c.3.1)$, $(c.3.2)$, and $(c.3.3)$.\\
\underline{Subsubcase $(c.3.1)$} Let $k=1$. If $h=1$, then $(k+pn,h+pn,1) = (pn+1,pn+1) \in \mathcal{P}_{1,2}$. This contradicts 
(\ref{conditionpnpn}). We assume that $h \geq 2$. Then, $(pn+1,pn+1) \in M_{pass5} \cap \mathcal{P}_{1,2} \subset move((x,y,1)) \cap \mathcal{P}$.\\
\underline{Subsubcase $(c.3.2)$} If $h=1$, then we use a method that is similar to the one used in $(c.3.1)$.\\
\underline{Subsubcase $(c.3.3)$} Suppose that $k,h \geq 2$. We consider subsubsubcases $(c.3.3.1)$ and $(c.3.3.2)$.\\
\underline{Subsubsubcase $(c.3.3.1)$} If $k+h \leq p+1$, $(k-1)+(h-1) \leq p-1$. Then, $(1+pn,1+pn) \in M_{pass6} \cap \mathcal{P}_{1,2} \subset move((x,y,1)) \cap \mathcal{P}$.\\
\underline{Subsubsubcase $(c.3.3.2)$} If $k+h = p+2$, $(x,y,1) = (k+pn,h+pn,1) \in \mathcal{P}_{1,3}$. This contradicts (\ref{conditionpnpn})).
Therefore, 
we assume that $p+3 \leq k+h \leq 2p-2$. Since $1 \leq k+h -(p+2) \leq p-4$, there exist $u,v$ such that
$u,v \in \mathbb{Z}_{\geq 0}$, $1 \leq u+v \leq p-1$, $u \leq k$, $v \leq h$ and $(k-u)+(h-v) = p+2$.
Then, 
$(k-u+pn,h-v+pn) \in M_{pass6} \cap \mathcal{P}_{1,3} \subset move((x,y,1)) \cap \mathcal{P}$.\\
\underline{Case $(d)$} Let 
\begin{align}\label{casedeq}	
(x,y,1) = (k,h,1) \in \mathcal{N}_1
\end{align}
such that $k,h \in \mathbb{Z}_{\geq 0}$ and $k,h \leq p-1$. By (\ref{casedeq}), $(k,h,1) \notin \mathcal{P}_{1,1}$, and hence,
$k+h \neq p+1$.
We consider subcases $(d.1)$ and $(d.2)$.\\
\underline{Subcase $(d.1)$} 
If $k=0$ or $h=0$, then $(0,0,1) \in M_{pass4}$ or $M_{pass5}$. In other words, we reach $(0,0,1)$ by reducing the positive coordinate to zero. \\
\underline{Subcase $(d.2)$} 
We suppose that $k,h \geq 1$. We consider subsubcases $(d.2.1)$, $(d.2.2)$, and $(d.2.3)$.\\
\underline{Subsubcase $(d.2.1)$}
If $k+h \leq p-1$, then $(0,0,1) \in M_{pass6}$.\\
\underline{Subsubcase $(d.2.2)$}
If $k+h = p$, then $(k,h,0) \in M_{pass7} \cap \mathcal{P}_0$.\\
\underline{Subsubcase $(d.2.3)$}
We suppose that $p+1 \leq k+h \leq 2p-2$. By (\ref{casedeq}), $k+h \neq p+1$. Then, there exist $u,v \in \mathcal{N}$ such that $u \leq k$, $v \leq h$, $u+v \leq p-1$, and
$(k-u)+(h-v) = p+1$, and $(k-u,h-v) \in M_{pass6} \cap \mathcal{P}_{1,1}$\\
\underline{Case $(e)$} Let 
\begin{equation}\label{caseeeq}	
(x,y,1) = (k,h+np,1) \in \mathcal{N}_1
\end{equation}
such that $n \in \mathcal{N}$, $k,h \in \mathbb{Z}_{\geq 0}$ and $k,h \leq p-1$. 
We consider subcases $(e.1)$ and $(e.2)$.\\
\underline{Subcase $(e.1)$} 
If $k=0$, then $(0,0,1) \in M_{pass4}$.\\
\underline{Subcase $(e.2)$}
We suppose that $k \geq 1$. We consider subsubcase $(e.2.1)$ and $(e.2.2)$.\\
\underline{Subsubcase $(e.2.1)$}
If $n \geq 2$, then $2 \leq p-k+1 < h+np$. By reducing $h+np$ to $p-k+1$, we have $(k,p-k+1) \in M_{pass4} \cap \mathcal{P}_{1,1}$.\\
\underline{Subsubcase $(e.2.2)$}
We suppose that $n = 1$. We consider subsubsubcases $(e.2.2.1)$ and $(e.2.2.2)$.\\
\underline{Subsubsubcase $(e.2.2.1)$} 
If $h \geq 1$, then $p-k+1 < h+np = h+p$. By reducing $h+np$ to $p-k+1$, $(k,p-k+1) \in M_{pass4} \cap \mathcal{P}_{1,1}$.\\
\underline{Subsubsubcase $(e.2.2.2)$} 
We suppose that $h = 0$. If $k=1$, then $(k,h+np,1) = (1,p,1) \in \mathcal{P}_{1,1}$. This contradicts (\ref{caseeeq}).
Therefore, we assume that $k \geq 2$. Then, by reducing $k$ to $1$, $(1,p) \in M_{pass4} \cap \mathcal{P}_{1,1}$.\\
\underline{Case $(f)$} Let 
\begin{align*}\label{caseeeq2}	
(x,y,1) = (k+np,h,1) \in \mathcal{N}_1
\end{align*}

such that $n \in \mathcal{N}$, $k,h \in \mathbb{Z}_{\geq 0}$ and $k,h \leq p-1$. Then, we prove by the method that is similar to the one used in Case $(e)$.
\end{proof}

\begin{thm}\label{PNsetsforpassmove}
The sets $\mathcal{P}$ and $\mathcal{N}$ defined in Definition \ref{defofP0P1sets} are the sets of $\mathcal{P}$-positions and $\mathcal{N}$-positions of the Generalized Ry\={u}\={o} Nim for $p$ respectively.
\end{thm}

\begin{proof}
Suppose that we start the game from a position $\{x,y,z,pass\}\in \mathcal{P}=\mathcal{P}_0 \cup \mathcal{P}_1$.
If $(x,y,pass)=(x,y,0) \in \mathcal{P}_0$, then by Lemma \ref{ppositionp0}, $move((x,y,pass)) \subset \mathcal{N}_0$.
If $(x,y,pass)=(x,y,1) \in \mathcal{P}_1$, then by Lemma \ref{fromPtononP}, $move((x,y,pass)) \subset \mathcal{N}$.
Therefore, any option we take leads to a position $ (x^{\prime},y^{\prime},z^{\prime},pass^{\prime}) \in \mathcal{N}$.
If $ (x^{\prime},y^{\prime},z^{\prime},pass^{\prime}) \in \mathcal{N}_0$, then by Lemma \ref{ppositionp0}, our opponent can choose a proper option that leads to a position in $\mathcal{P}_0$. If $ (x^{\prime},y^{\prime},z^{\prime},pass^{\prime}) \in \mathcal{N}_1$, then by Lemma \ref{fromnonPtoP}, our opponent can choose a proper option that leads to a position in $\mathcal{P}$.
Note that any option reduces some of the numbers in the coordinates. In this way, our opponent can always reach a position in $\mathcal{P}$, and will finally win by reaching $\{0,0,0\}$ or $\{0,0,1\} \in \mathcal{P}$. 
Therefore, $\mathcal{P}$ is the set of $\mathcal{P}$-position.
If we start the game from a position in $\mathcal{N}$, then Lemma \ref{ppositionp0} and Lemma \ref{fromnonPtoP} mean that we can choose a proper option that leads to a position in $ \mathcal{P}$. Any option from this position taken by our opponent leads to a position in $\mathcal{N}$. In this way, we win the game by reaching $\{0,0,0\}$ or $\{0,0,1\}$. Therefore, $\mathcal{N}$ is the set of $\mathcal{N}$-positions.
\end{proof}

\section{Generalized {\boldmath Ry\={u}\={o}} (dragon king) Nim that Restricted the Diagonal and Side Movement}

Restrict the diagonal movement by $p \in \mathbb{Z}_{>1}$ and side movement by $q \in \mathbb{Z}_{>1}$. It is possible to take up to a total of $p$ tokens when taking them at once and up to $q$ tokens when taking them from one heaps.

\begin{figure}[!htb]
\begin{center}
\includegraphics[width=0.45\columnwidth,bb=0 0 253 239]{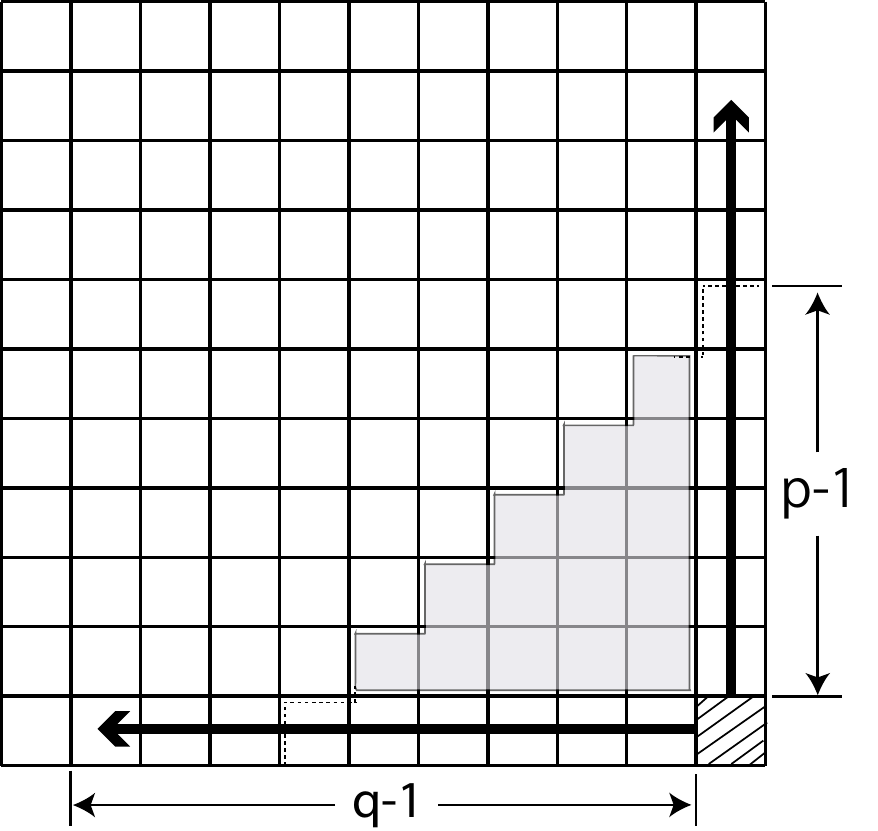}
\caption{Restricted the Diagonal and Side Movement (1)}
\label{suetsugu}
\end{center}
\end{figure}

In this case, Grundy Number is known only in the following cases:
\begin{thm}
If $\bmod(q,p) = 0$, then we have
\begin{align*}
\mathcal{G}(x,y)=\bmod(\bmod(x,q)+ \bmod(y,q), p)
+p(\lfloor\frac{\bmod(x,q)}{p}\rfloor\oplus\lfloor\frac{\bmod(y,q)}{p}\rfloor).
\end{align*}
\end{thm}

\begin{thm}
If $\bmod(q,p) = 1$, then we have

\begin{align*}
\mathcal{G}(x,y) = \begin{cases}
q \ \ \ \ \ (where \bmod(x,q) = 0, \bmod(y,q) = 0, x\neq 0, y\neq 0) \\
\bmod(\bmod(x,q)+\bmod(y,q),p)
+p(\lfloor\frac{\bmod(x,q)}{p}\rfloor\oplus\lfloor\frac{\bmod(y,q)}{p}\rfloor). \\
\hspace{80mm} (otherwise)
\end{cases}
\end{align*}

\end{thm}
In other case, it becomes complicated and generally difficult.\\

Next, instead of the restiriction of side movement by $q \in \mathbb{Z}_{>1}$, we consider the restriction of horizontal movement by $q \in \mathbb{Z}_{>1}$ and the restriction of vertical movement by $r \in \mathbb{Z}_{>1}$.


\begin{figure}[!htb]
\begin{center}
\includegraphics[width=0.45\columnwidth,bb=0 0 276 239]{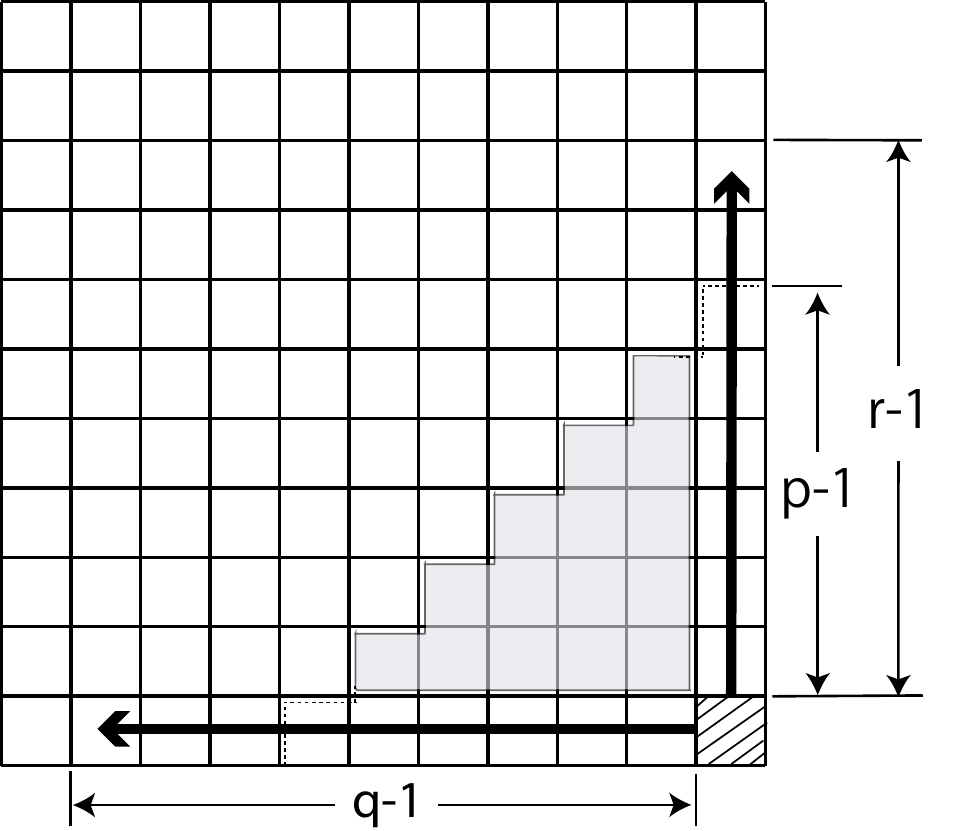}
\caption{Restricted the Diagonal and Side Movement (2)}
\label{suetsugu1}
\end{center}
\end{figure}

\begin{thm}
If $\bmod(q,p) = 0$ \ and \ $\bmod(r,p) = 0$, then we have
\begin{align*}
\mathcal{G}(x,y)=\bmod(\bmod(x,q)+\bmod(y,r),p)
+p(\lfloor\frac{\bmod(x,q)}{p}\rfloor\oplus\lfloor\frac{\bmod(y,r)}{p}\rfloor).
\end{align*}
\end{thm}

\section{n-dimensional Ry\={u}\={o} Nim}

\begin{defn}[$3$-dimensitonal Ry\={u}\={o} Nim]
$3$-dimensional Ry\={u}\={o} Nim is an impartial game with three heaps of tokens. The rules are as follows:
\begin{itemize}
\item\ The legal move is to remove any number of tokens from a single heap (as in Nim) or 
\item\ remove one token from any two heaps or remove one token from all the three heaps.
\item\ The end position is the state of no tokens in the three heaps.
\end{itemize}
\end{defn}

We could not get the indication of Grundy Number for 3-dimensional Ry\={u}\={o} Nim but we get the $\mathcal{P}$-positions as shown in this theorem.

\begin{thm}
Let $(x,y,z)$ be a 3-dimensional Ry\={u}\={o} Nim position. 
The $\mathcal{P}$-positions of 3-dimensional Ry\={u}\={o} Nim are given as follows:

$\bmod(x+y+z,3) = 0$, and moreover
\begin{itemize}
\item[(A)] If $\bmod(x,3) = \bmod(y,3) = \bmod(z,3) = 1$ then
\begin{align*}
\lfloor \frac{x}{3} \rfloor \oplus \lfloor \frac{y}{3} \rfloor \oplus \lfloor \frac{z}{3} \rfloor \oplus 1 = 0
\end{align*}
\item[(B)] Otherwise
\begin{align*}
\lfloor \frac{x}{3} \rfloor \oplus \lfloor \frac{y}{3} \rfloor \oplus \lfloor \frac{z}{3} \rfloor =0.
\end{align*}
\end{itemize}
\end{thm}

\begin{defn}[Modified $3$-dimensional Ry\={u}\={o} Nim]
 We also consider the modified rule of 3-dimensional Ry\={u}\={o} Nim as follows:
\begin{itemize}
\item\ The legal move is to remove any number of tokens from a single heap (as in Nim) or 
\item\ remove one token from any two heaps or 
\item\ \sout{remove one token from all the three heaps}.
\item The end position is the state of no tokens in the three heaps. 
\end{itemize}
\end{defn}

Then, we can obtain the Grundy Number of this game as follows:
\begin{thm}
Let $(x,y,z)$ be a Modified $3$-dimensional Ry\={u}\={o} Nim position, then we have

\begin{align*}
\mathcal{G}(x,y,z)=\bmod(x+y+z,3)+3(\lfloor \frac{x}{3} \rfloor \oplus \lfloor \frac{y}{3} \rfloor \oplus \lfloor \frac{z}{3} \rfloor).
\end{align*}
\end{thm}

Next, we also consider $n$-dimensional Ry\={u}\={o} Nim.

\begin{defn}[$n$-dimensional Ry\={u}\={o} Nim]
$n$-dimensional Ry\={u}\={o} Nim is an impartial game with $n$ heaps of tokens. The rules are as follows:
\begin{itemize}
\item The legal move is to remove any number of tokens from a single heap (as in Nim) or
\item 
The total number of tokens removed from the $k \in \mathbb{Z} (1<k\leq n)$ heaps at once must be less than $p \in \mathbb{Z}_{>1}$.
\item 
The end position is the state of no tokens in the $n$ heaps.
\end{itemize}
\end{defn}

\begin{thm}
Let $(x_1,\ldots,x_n)$ be a $n$-dimensional Ry\={u}\={o} Nim position, then we have
\begin{align*}
\mathcal{G}(x_1,\ldots,x_n)= \bmod(x_1+\cdots+x_n,p)+p(\lfloor \frac{x_1}{p} \rfloor \oplus \cdots \oplus \lfloor \frac{x_n}{p} \rfloor). \ (p \in \mathbb{Z}_{>1}).
\end{align*}
\end{thm}

\end{document}